\documentclass{article}

\usepackage[T1]{fontenc}
\usepackage[utf8]{inputenc}
\usepackage[italian,english]{babel}

\DeclareUnicodeCharacter{2212}{-}

\usepackage{geometry}
\usepackage{emptypage}	
\usepackage{graphics}
\usepackage{graphicx}
\usepackage{amsmath}
\usepackage{amssymb}
\usepackage{amsthm}
\usepackage{graphicx}
\usepackage{wrapfig}
\usepackage{subfig}
\usepackage[font=footnotesize]{caption}
\usepackage{pdfpages}
\usepackage{braket}
\usepackage{enumerate}
\usepackage{hyperref}
\usepackage{mathtools}

\theoremstyle{plain}
\newtheorem{thm}{Theorem}[section]
\newtheorem{thm*}{Theorem}

\newtheorem{prop}[thm]{Proposition}

\theoremstyle{definition}
\newtheorem{dfn}[thm]{Definition}

\newtheorem{rmk}[thm]{Remark}

\newcommand{\R}{\mathbb{R}}
\newcommand{\N}{\mathbb{N}}
\newcommand{\s}{\mathcal{S}}
\newcommand{\abs}[1]{\lvert #1 \rvert}
\newcommand{\defeq}{\coloneqq}

\title{Bifurcations of planar balanced configurations for the $n$-body problem in $\R^4$}
\author{
Luca Asselle, Katharina Kormann, Giorgia Testolina\\[3pt]
\textit{Ruhr-Universität Bochum, Fakultät für Mathematik}
}

\date{\today}

\begin{document}

\maketitle

\begin{abstract}
\small
Central configurations play a fundamental role in the Newtonian $n$-body problem, as they generate motions that preserve their shape up to rotation and scaling. These include relative equilibria, in which the bodies rigidly rotate about the center of mass along circular orbits. For $d\le3$, such motions arise only from planar central configurations, whereas in higher dimensions the richer structure of the orthogonal group allows for balanced configurations giving rise to non-planar relative equilibria. In this work, we study bifurcations of balanced configurations in $\mathbb{R}^4$ from the trivial branch generated by a planar central configuration. The main difficulty is that, even after quotienting out the rotational symmetry, the corresponding critical point may remain degenerate. To address this issue, we establish a bifurcation criterion for a parameter-dependent family of involution-invariant functionals, whose restriction to the fixed-point set is independent of the parameter. If the relevant critical point is isolated and homologically visible, a change in the Morse index of the normal Hessian forces bifurcation. We then apply this criterion to the balanced configuration problem, obtaining a lower bound on the number of bifurcation points from the planar branch in terms of the spectrum of the normal Hessian. Finally, using numerical computations, we explicitly identify bifurcation branches for the four and five body problems that detach from the bifurcation points guaranteed by the theorem.
\end{abstract}

\section{Introduction}
\label{section_1}

The $n$-body problem is a cornerstone of celestial mechanics, which studies the motion of $n$ point masses under their mutual gravitational attraction. 
Let $m_1, \dots, m_n > 0$ be $n$ point masses with positions $q_1, \dots, q_n \in \R^d$, $d \geq 2$.
Introducing the configuration vector $q \defeq (q_1, \dots, q_n) \in \R^{nd}$, the diagonal mass matrix $M \defeq \text{diag}(m_1 \mathbb{I}_d, \dots, m_n \mathbb{I}_d)$, and the Newtonian gravitational potential
\[
U(q) \defeq \sum_{i < j} \frac{m_i m_j}{\abs{q_i - q_j}},
\]
Newton's equations of motion are given by
\begin{equation}
\label{eqn_motion}
\ddot{q} = M^{-1} \nabla U(q).
\end{equation}
System~\eqref{eqn_motion} is invariant under the symmetries of the Euclidean group, which, by Noether’s theorem, correspond to conserved quantities. In particular, invariance under translations implies conservation of total linear momentum. As a consequence, the center of mass
\[
\bar{q} = \frac{1}{\bar{m}} \sum_{i=1}^n m_i q_i, \quad \bar{m} = \sum_{i=1}^n m_i,
\]
moves with constant velocity and can therefore be fixed at the origin without loss of generality.
Under this condition, and excluding collisions, the collision-free configuration space is defined by
\[
\mathcal {C} \defeq \{ q \in \R^{dn} \mid \bar{q} = 0, \, q_i \neq q_j \text{ for } i \neq j \}.
\]
For $n \geq 3$, system~\eqref{eqn_motion} is extremely difficult to solve, and in fact a complete explicit solution is known only in the case $n = 2$. A classical approach --- pioneered by Euler and Lagrange and further developed by Albouy, Chenciner, Conley, Moeckel, Montgomery, and Smale, among others --- is therefore to look for special configurations of the masses that give rise to simple types of motion. Among these, \emph{central} and \emph{balanced configurations}, play a fundamental role~\cite{Albouy_Chen, AP, Moeckel1994, Moeckel2014}.

A central configuration is a special arrangement of the masses such that the acceleration vector of each body points towards the center of mass, with magnitude proportional to its distance from the center of mass. Mathematically, this condition can be expressed as

\begin{equation}
\label{eqn:CC}
M^{-1} \nabla U(q) + \lambda q = 0,
\end{equation}
for some positive constant $\lambda$.
The study of central configurations dates back to Euler (1767) and Lagrange (1772), who classified all central configurations for the three-body problem. For any choice of masses, there are three classes of collinear central configurations, corresponding to different orderings of the masses along the line, and two planar configurations, one for each orientation, where the masses form the vertices of an equilateral triangle. For more information about central configurations, we refer to the lectures by Moeckel~\cite{Moeckel1994, Moeckel2014}.

As already mentioned, one of the features that make the study of central configurations important, is that they yield simple solutions of the $n$-body problem. For $d \leq 3$, every planar central configuration can generate a \emph{homographic motion}, in which the configuration remains similar to the initial one throughout the motion, while each body follows a Keplerian orbit. A special case is the \emph{relative equilibrium}, which correspond to circular orbits: in this situation, the configuration rotates rigidly about the center of mass with constant angular velocity. On the other hand, spatial configurations can only generate \emph{homothetic motions}, in which the shape is preserved up to scaling as the system expands from or collapses into the center of mass.

The situation changes significantly when $d \geq 4$. Albouy and Chenciner~\cite{Albouy_Chen} described new types of balanced configurations, which arise due to the richer structure of the orthogonal group in \(\mathbb{R}^d\). This structure allows, for example, simultaneous rotations in two mutually orthogonal planes with different angular velocities. As a consequence, relative equilibria are no longer restricted to planar configurations, and the motion does not necessarily remain planar.

In this article we focus on balanced configurations in $\R^4$. For a detailed discussion of balanced configurations, we refer the reader to~\cite{AP, Moeckel2014}. Let $s > 1$ and define the matrices
\[
S \defeq \mathrm{diag}(s, s, 1, 1), \qquad \hat{S} \defeq \mathrm{diag}(S, \dots, S).
\]
A balanced configuration $q\in \R^{4n}$ is said to be balanced if it satisfies
\begin{equation}
\label{eqn:SBC}
M^{-1} \nabla U(q) + \lambda \hat{S} q = 0,
\end{equation}
for some $\lambda > 0$. This equation is invariant under the $SO(2) \times SO(2)$ action given by independent rotations in the $\R^2 \times \{(0,0)\}$ and $\{(0,0)\} \times \R^2$ planes. It follows that balanced configurations are never isolated: they occur in $S^1$ families if they lie in one of the two coordinate planes, and in $S^1 \times S^1$ families otherwise.

A fundamental property of central and balanced configurations is their
variational characterization as critical points of the Newtonian potential
$U$ under suitable normalization constraints. This variational viewpoint
allows one to apply bifurcation theory, where varying a physical parameter
may lead to new families of solutions. In finite-dimensional variational
problems, a classical bifurcation criterion asserts that a change of the
Morse index along a branch of critical points with non-degenerate endpoints
forces bifurcation.
In~\cite{APF}, this approach was used to study bifurcations of collinear
central configurations in $\R^4$. Such configurations form a trivial
branch of balanced configurations for every $s>1$, and after a suitable
symmetry reduction the critical points become non-degenerate except at
finitely many parameter values. The variation of their Morse index then
detects bifurcations to non-collinear balanced configurations.
In this work, we extend the analysis to planar central configurations in
$\{(0,0)\}\times\R^2$. These configurations again generate trivial
branches of balanced configurations, but a new difficulty appears:
after quotienting out the rotational symmetry, the corresponding planar
critical point may still be genuinely degenerate. Moreover, since the
planar block of the Hessian is independent of $s$, this degeneracy
persists along the whole branch. Thus the classical bifurcation criterion
cannot in general be applied.
To overcome this difficulty, we prove an abstract bifurcation theorem for
persistently degenerate branches. The framework is motivated by a natural
involution of the balanced configuration problem, whose fixed-point set is
the space of planar configurations. The restriction of the functional to
this fixed-point set is independent of the parameter, while the Hessian in
the normal directions varies with $s$. Critical groups are then used to
retain enough local information at a possibly degenerate planar critical
point. More precisely, let $(F_\lambda)_{\lambda\in I}$ be a family of
functionals invariant under an involution $\tau$, and let
$P=\operatorname{Fix}(\tau)$.
Assume that the restriction of $F_\lambda$ to $P$ is independent of
$\lambda$, and denote it by $f$. If $q\in P$ is a critical point of
$f$, then the symmetry implies that $q$ is a critical point of
$F_\lambda$ for every $\lambda$, and hence determines the constant
branch of critical points $\gamma(\lambda)\equiv q$.
The Hessian along this branch splits into a tangential part, given by
the Hessian of $f$, and a normal part, which may vary with $\lambda$.
We call the latter the \emph{normal Hessian}. Our result shows that,
provided $q$ is isolated and homologically visible, a change of the
Morse index of the normal Hessian forces bifurcation.

\begin{thm}\label{thm_1}
Let $(F_\lambda)_{\lambda\in I}$ be a $\mathcal C^2$ family of
functionals invariant under a smooth involution $\tau$, and assume that
the restriction to $P=\operatorname{Fix}(\tau)$ is independent of
$\lambda$. Let $q\in P$ be an isolated homologically visible critical
point of this restriction. If the normal Hessians at the endpoints of
$I=[a,b]$ are non-degenerate and have different Morse indices, then
there exists a bifurcation instant in $(a,b)$ from the constant branch
$\gamma(\lambda)\equiv q$.
\end{thm}

Theorem~\ref{thm_1} is related to several variational approaches to
bifurcation along branches of critical points. In the classical
spectral-flow framework, see for instance~\cite{FPR}, bifurcation is
detected by a non-vanishing spectral flow of the Hessians, provided that
the endpoint Hessians are non-degenerate.
In degenerate settings, bifurcations have been studied for instance by equivariant
methods. In particular, equivariant degree and equivariant Conley index
have been used to study bifurcations from branches of critical orbits,
with applications to the $N$-body problem; see, e.g.,
\cite{Kowalczyk2015}. More recently, isolated degenerate critical orbits
have been treated by means of the equivariant gradient degree in
\cite{GKS2023}. A different approach, closer in spirit to the one used
here, is developed by Lu~\cite{LuParameterized}, using parameterized
splitting theorems and the continuation properties of critical groups.
These results do not seem to yield Theorem~\ref{thm_1} directly under
the hypotheses relevant to the $S$-balanced problem. In our setting, the
tangential part of the Hessian may remain degenerate for every value of
the parameter, while the restriction to the fixed-point set is
parameter-independent and carries non-trivial local homological information.
Bifurcation is then detected solely through a change of the Morse index
in the normal directions.

For the balanced configuration problem, the normal Hessian can be computed
explicitly. Let $\hat q$ be a planar central configuration and denote by
\[
\mu_k(\hat q)<\cdots<\mu_l(\hat q)=-U(\hat q)
<\cdots<\mu_1(\hat q)<0
\]
the distinct nonzero eigenvalues of the normal block introduced in
Section~\ref{sec_2.1}. The possible changes of the normal Morse index occur
at
\[
s_i=-\frac{\mu_i(\hat q)}{U(\hat q)},
\qquad i=l+1,\ldots,k.
\]
The abstract result above gives the following.

\begin{thm}\label{thm_2}
Let $\hat q$ be a non-collinear planar central configuration and let
$q=[\hat q]$ denote its class modulo planar rotations. Assume that $q$ is
an isolated homologically visible critical point of the restriction of
$U$ to the space of planar configurations modulo rotations. Then every
$s_i$ defined above, $i=l+1,\ldots,k,$
is a bifurcation instant from the planar branch generated by $\hat q$.
In particular, there are at least $k-l$ distinct bifurcation instants,
and the bifurcating configurations are non-planar.
\end{thm}

The paper is organized as follows. In Section~2, we recall the definition
and main properties of balanced configurations in $\R^4$. In Section~3,
we analyze planar balanced configurations and compute the variation of
their normal inertia indices with the parameter $s$. Section~4 contains
the abstract bifurcation theorem for persistently degenerate branches.
In Section~5, we apply this result to planar balanced configurations and
prove Theorem~\ref{thm_2}. Finally, Sections~6 and~7 contain
numerical simulations. Section~6 illustrates several non-trivial branches
bifurcating from planar central configurations for $n=4$ and $n=5$,
whereas Section~7 extends the numerical study of collinear configurations
from~\cite{APF} to the four-body problem.

\vspace{2mm}

\noindent \textbf{Acknowledgements.} We would like to thank  Alessandro Portaluri for the insightful discussions that contributed to the development of this work. L.A. is partially funded by the DFG project 566804407 ``Symplectic Dynamics, Celestial Mechanics and Magnetism''. G.T. is partially funded by the CRC/TRR 191 ``Symplectic structures in Geometry, Algebra and Dynamics'' and by the DFG project 566804407 ``Symplectic Dynamics, Celestial Mechanics and Magnetism''.


\section{Balanced configurations: definition and main properties}
\label{section_2}

In this section, we recall the definition of balanced configurations in $\R^4$ and summarize their main properties. For further details see~\cite{AP, Moeckel2014}.
Consider $n$ point masses $m_1, \dots, m_n > 0$ with positions $q_1, \dots, q_n \in \R^4$ moving according to Newton's law of gravitation 

\begin{equation}
\label{eqn_Newton}
\ddot{q} = M^{-1} \nabla U(q)
\end{equation}
where $q \defeq (q_1, \dots, q_n) \in \R^{4n}$ and $M \defeq \text{diag}(m_1 \mathbb{I}_4, \dots, m_n \mathbb{I}_4)$. As before, we restrict our attention to the collision-free configuration space with center of mass at the origin

\[
\mathcal{C} = \{ q \in \R^{4n} \mid \bar{q} = 0, \, q_i \neq q_j \, \text{for} \, i \neq j \}.
\]
Let $s > 1$ be a positive real number, and define the matrices 

\[
S_s \defeq \text{diag}(s, s, 1, 1) \in \R^{4 \times 4},
\quad \hat{S}_s \defeq \text{diag}(S_s, \dots, S_s) \in \R^{4n\times 4n}.
\]

\begin{dfn} 
\label{dfn_SBC}
A balanced configuration is an arrangement of the masses such that 
\begin{equation}
\label{eqn_SBC}
M^{-1} \nabla U(q) + \lambda \hat{S}_s q = 0,
\end{equation}
where $\lambda > 0$ is a positive constant.
\end{dfn}

\begin{rmk}
As pointed out by Asselle and Portaluri~\cite{AP}, this definition is equivalent to that of Albouy and Chenciner~\cite{Albouy_Chen}, also used by Moeckel~\cite{Moeckel2014}, where $S = -A^2$ with $A$ a skew-symmetric matrix. Since the problem is invariant under orthogonal changes of coordinates, one may always assume that $S$ is diagonal, by replacing the canonical basis of $\R^4$ with an orthonormal spectral basis of $S$. Finally, the scaling invariance of the problem allows us to take $S$ to be of the form above.
\end{rmk}

A balanced configuration $q$ gives rise to a relative equilibrium 
\[
q(t) \defeq \begin{pmatrix}
e^{i \sqrt{s\lambda} t}& 0 \\
0 & e^{i \sqrt{\lambda} t}
\end{pmatrix} 
\cdot q
\]
in which the two orthogonal coordinate planes rotate with angular velocities $\sqrt{s\lambda}$ and $\sqrt{\lambda}$. For $s=1$ one recovers the relative equilibria generated by central configurations. We note that, in the balanced case, the motion need not be periodic: if $q$ has nonzero components in both coordinate planes, it is quasi-periodic whenever $\sqrt s$ is irrational. This reflects the fact that the central configuration problem is invariant under the whole $SO(4)$-group, whereas for $s>1$ the balanced configuration problem is invariant only under the $SO(2) \times SO(2)$-action given by independent rotations in the two coordinate planes. In particular, balanced configurations are never isolated: they occur in $S^1$-families when they lie in one of the two coordinate planes, and in $S^1 \times S^1$-families otherwise. We denote by
$\langle \cdot, \cdot \rangle_M \defeq \langle M\cdot, \cdot\rangle$
the mass scalar product on the configuration space. Thus, the gradient of $U$ with respect to the mass scalar product is
\[
\nabla_M U(q) = M^{-1}\nabla U(q).
\]
We define the weighted scalar product $\langle \cdot, \cdot \rangle_S$ as

\[
\langle \cdot, \cdot \rangle_{S_s} \defeq \langle \hat{S}_s M \cdot, \cdot \rangle.
\]
Observing that the Newtonian potential $U(q)$ is homogeneous of degree $-1$, and applying Euler's theorem, we can determine the value of the constant $\lambda$ in equation~\eqref{eqn_SBC}:
\[
\lambda = \frac{U(q)}{\abs{q}^2_{S_s}}.
\]

It follows from Definition~\ref{dfn_SBC} that if $q$ is a balanced
configuration, then $kq$ is also balanced for every
$k>0$. Therefore, we can restrict our attention to the collision-free
configuration sphere
\[
\mathcal S_s
\defeq
\left\{
q\in\mathcal C
\ \middle|\
|q|_{S_s}^2=1
\right\}.
\]
The manifold $\s_s$ is an open subset of a smooth compact manifold diffeomorphic to a $(4n - 5)$-dimensional sphere. Moreover, since in $\s_s$ we have $\lambda = U(q)$, equation~\eqref{eqn_SBC} becomes

\begin{equation}
\label{eqn_SBC_n}
M^{-1} \nabla U(q) + U(q) \hat{S}_s q = 0.
\end{equation}

Finally, we describe the variational characterization of balanced configurations, a key feature that will play a fundamental role in our discussion. The proof of this theorem can be found in~\cite{Moeckel2014}.

\begin{thm}
A configuration $q\in \mathcal S_s$ is  balanced if and only if it is a critical point of $U|_{\s_s}$.
\end{thm}

\noindent 
Thus, the problem of finding balanced configurations is equivalent to that of
finding critical points of $U|_{\mathcal S_s}$. The Hessian of
$U|_{\mathcal S_s}$
at a critical point $q$ is the quadratic form on $T_q\mathcal S_s$
represented, with respect to the Euclidean scalar product, by the
$(4n\times 4n)$ matrix
\[
\widehat H_s(q)
=
D^2U(q)+\widehat S_s M U(q).
\]
The idea behind this formula is to construct a quadratic form on
$T_q\R^{4n}$ whose restriction to $T_q\mathcal S_s$ agrees with the
Hessian of $U|_{\mathcal S_s}$. To this end, one considers the function
on $\R^{4n}$ obtained by multiplying $U$ by the normalization factor
$\lvert q\rvert_{S_s}$. On $\mathcal S_s$, this factor is identically
equal to one, so the resulting function agrees with $U|_{\mathcal S_s}$.
Moreover, its first derivative in directions tangent to $\mathcal S_s$
vanishes at $q$, while its second derivative contributes the additional
term $\widehat S_s M U(q)$. A detailed proof can be found
in~\cite[Proposition~16]{Moeckel2014}.
Equivalently, with respect to the mass scalar product, the same quadratic
form is represented by
\[
H_s(q)
=
M^{-1}D^2U(q)+\widehat S_s U(q).
\]
We shall use this latter representation throughout the paper.

A straightforward computation shows that the $(i, j)$-th entry of $D^2 U(q)$ is given by 

\begin{align*}
D_{ij} & = \frac{m_i m_j}{r^3_{ij}} (\mathbb{I}_4 - 3 u_{ij} u^t_{ij}) \quad \text{for} \, i \neq j \\
D_{ii} & = - \sum_{j \neq i} D_{ij}
\end{align*}
where $r_{ij} \defeq \abs{q_i - q_j} $ and $u_{ij} \defeq \frac{q_i - q_j}{\abs{q_i - q_j}}$. 
The rotational invariance described above implies that balanced configurations are never isolated critical points of $U|_{\s_s}$ and that the Hessian is always degenerate. More precisely, we have
\[
T_q\big((SO(2)\times SO(2))\cdot q\big)
\subseteq
\ker \operatorname{Hess}\big(U|_{\mathcal S_s}\big)(q).
\]
We say that a balanced configuration is \textit{Morse-Bott non-degenerate} if  the only degeneracy of the Hessian is that induced by rotational symmetry, namely if equality holds above.


\section{Morse index of planar balanced configurations}
\label{sec_2.1}

The aim of this section is to analyze how the inertia indices of planar
balanced configurations lying in the plane $\{(0,0)\}\times\R^2$ depend
on the parameter $s$. Throughout this section, we assume that the planar
configuration under consideration is non-collinear, since the collinear
case has already been treated in~\cite{APF}. To reduce the effect of
rotational symmetry, we restrict our attention to configurations contained
in $\{0\}\times\R^3\subset\R^4$. The associated configuration space is
\[
\mathcal{S}_s^\text{red} = \left\{ q \in \mathcal{S}_s \; \middle| \; q_i \in \{ 0 \} \times \R^3 \; \forall \; i = 1, \dots, n \right\} \subset \mathcal{S}_s.
\]
Identifying $\{ 0 \} \times \R^3$ with $\R^3$, the matrix $S_s$ in the balanced configuration equation reduces to
\[
S_s = \mathrm{diag}(s, 1, 1).
\]
Hence, balanced configurations occur in $S^1$-families corresponding to
the $SO(2)$-action generated by rotations in the plane
$\{0\}\times\R^2$. This action is free on
$\mathcal S_s^{\mathrm{red}}\setminus\mathcal C_{\mathrm{coll}}$, where $\mathcal C_{\mathrm{coll}}$ denotes the set of collinear configurations along the $e_1$-axis. 
Therefore, in order to remove the degeneracy induced by the $S^1$-symmetry,
we may pass to the smooth quotient
\[
\overline{\mathcal S}_s
\defeq
\big(
\mathcal S_s^{\mathrm{red}}\setminus\mathcal C_{\mathrm{coll}}
\big)/SO(2).
\]
This quotient will be used in the bifurcation argument in Chapter \ref{sec_2.3}. For the explicit
computation of the Hessian, however, it is more convenient to work on
$\mathcal S_s^{\mathrm{red}}$.

\begin{rmk}\label{rmk_45}
For the local bifurcation problem considered in this paper, passing to the
quotient above causes no difficulty, since we work in a neighborhood of a
fixed planar configuration where the $SO(2)$-action is free. The situation is more delicate if one wants to develop Morse theory on the
quotient, since the singular set $\mathcal C_{\mathrm{coll}}$ has been
removed. In that context, one needs to know that this removal does not lose
relevant information. This is the role played by the $45^\circ$-Theorem: $\mathcal C_{\mathrm{coll}}$ is an attractor for the
projectivized gradient flow, and in particular no non-collinear balanced
configurations occur sufficiently close to it; see~\cite[Section~5]{AP}.
\end{rmk}

A further symmetry will play an important role in the bifurcation
analysis. Set $\tau:\R^3\to \R^3, \tau(x,y,z) := (-x,y,z)$, and denote the induced involution on $\mathcal S_s^{\mathrm{red}}$ again by $\tau$:
\[
\tau:\mathcal S_s^{\mathrm{red}}\longrightarrow
\mathcal S_s^{\mathrm{red}},
\qquad
\tau(q_1,\ldots,q_n)
=
(\tau q_1,\ldots,\tau q_n).
\]
Since the Newtonian potential depends only on mutual distances, it is
invariant under $\tau$. Moreover, $\tau$ commutes with the $SO(2)$-action
given by rotations in the $(y,z)$-plane. It therefore induces an
involution, still denoted by $\tau$, on the quotient
$\overline{\mathcal S}_s$.
Its fixed-point set is precisely the space of planar configurations modulo
rotations:
\[
\mathcal P
\defeq
\text{Fix}(\tau)
=
\left\{
q\in\mathcal S_s^{\mathrm{red}}
\ \middle|\
q_i\in\{0\}\times\R^2,\quad i=1,\ldots,n
\right\}\Big/SO(2).
\]
Notice that $\mathcal P$ does not dependent on $s$, since on $\mathcal P$ all the $x$-coordinates vanish. Consequently, the restriction $$f\defeq U|_{\mathcal P}$$ is independent of $s$ as well.
Although the $\overline{\mathcal S}_s$ depend on the
parameter $s$, they can be naturally identified. Indeed, define
\[
\Phi_s:\mathcal S_1^{\mathrm{red}}
\longrightarrow
\mathcal S_s^{\mathrm{red}}, \quad 
\Phi_s(x_i,y_i,z_i)
=
\left(\frac{x_i}{\sqrt{s}},y_i,z_i\right).
\]
Then $\Phi_s$ is a diffeomorphism, it commutes with both the $SO(2)$-action
and the involution $\tau$, and it restricts to the identity on the planar
fixed-point set $\mathcal P$. Hence it descends to a diffeomorphism
between the corresponding quotient spaces.

With these preliminaries in place, we proceed to analyze planar balanced
configurations in the plane $\{0\}\times\R^2$. Let $\hat q$ denote such
a configuration. Then $\hat q$ is a central configuration and hence
defines a trivial branch of balanced configurations for all $s>1$.
Rearranging the coordinates of $\hat q$ as
$(x_1,\ldots,x_n,y_1,\ldots,y_n,z_1,\ldots,z_n)$,
the Hessian $H_s(\hat q)$ decomposes as
\[
H_s(\hat q)
=
\begin{pmatrix}
M^{-1}B(\hat q)+sU(\hat q)I_n & 0\\
0 & M^{-1}D(\hat q)+U(\hat q)I_{2n}
\end{pmatrix}.
\]
Here:
\begin{itemize}
    \item $B(\hat q)$ is the $n\times n$ matrix whose entries are given by
    \[
    b_{ij}
    =
    \frac{m_i m_j}{r_{ij}^3},
    \qquad
    b_{ii}
    =
    -\sum_{j\neq i}b_{ij};
    \]

    \item $D(\hat q)$ is the $2n\times2n$ block matrix whose blocks are
    given by
    \[
    D_{ij}
    =
    \frac{m_i m_j}{r_{ij}^3}
    \begin{pmatrix}
    1-3\cos^2\theta_{ij}
        & -3\sin\theta_{ij}\cos\theta_{ij}\\
    -3\sin\theta_{ij}\cos\theta_{ij}
        & 1-3\sin^2\theta_{ij}
    \end{pmatrix},
    \qquad
    D_{ii}
    =
    -\sum_{j\neq i}D_{ij},
    \]
    where $\theta_{ij}$ is the angle between $q_i-q_j$ and the
    $y$-axis in $\R^3$.
\end{itemize}

\noindent
We denote the normal and planar blocks of the Hessian respectively by
\[
A_s(\hat q)
\defeq
M^{-1}B(\hat q)+sU(\hat q)I_n, \quad 
P(\hat q)
\defeq
M^{-1}D(\hat q)+U(\hat q)I_{2n}.
\]
Notice that the planar block $P(\hat q)$ is independent of the parameter
$s$. The normal block is understood as restricted to the space of variations orthogonal to the plane of $\hat q$
and compatible with the center-of-mass constraint. Accordingly, we denote by
\[
\big(
\iota_n^-(\hat q;s),
\iota_n^0(\hat q;s),
\iota_n^+(\hat q;s)
\big)
\]
the inertia indices of the normal block $A_s(\hat q)$, and by
\[
\big(
\iota_p^-(\hat q),
\iota_p^0(\hat q),
\iota_p^+(\hat q)
\big)
\]
the inertia indices of the planar block $P(\hat q)$.

\begin{rmk}
At $s=1$, the rotational invariance of the central configuration problem
implies
\[
\iota_n^0(\hat q;1)\geq2,
\qquad
\iota_p^0(\hat q)\geq1.
\]
The two normal zero directions are induced by rotations moving the plane
containing $\hat q$, whereas the planar zero direction is induced by
rotations within that plane. Equality holds when the
only degeneracies are those induced by rotational symmetry; see
~\cite[p.~67]{Moeckel1994}. Since $P(\hat q)$ does not depend on $s$, its nullity is constant along
the trivial branch. The normal nullity, on the other hand, depends on $s$
and vanishes except at the finitely many parameter values described below.
\end{rmk}

From the decomposition of the Hessian above, the planar inertia indices
are independent of $s$. In particular,
$\iota_p^0(\hat q)\geq 1$,
so that $\hat q$ is a degenerate critical point of
$U|_{\mathcal S_s^{\mathrm{red}}}$ for every $s>1$.
We now describe the normal inertia indices in terms of the spectrum of
$M^{-1}B(\hat q)$. For the spectral properties used below, we refer to
\cite[pp.~56--60]{Moeckel1994}. The $n\times n$ matrix
$M^{-1}B(\hat q)$ has a zero eigenvalue whose eigenvector is
$(1,\ldots,1)$. This eigenvalue does not contribute to the inertia
indices, since the corresponding direction is transverse to
$T_{\hat q}\mathcal S_s^{\mathrm{red}}$.
Among the remaining $n-1$ eigenvalues, at least two are equal to
$-U(\hat q)$. Moreover, if $n>3$, there is at least one eigenvalue
strictly smaller than $-U(\hat q)$. We denote the distinct nonzero
eigenvalues of $M^{-1}B(\hat q)$ by
\[
\mu_k(\hat q)
<
\cdots
<
\mu_l(\hat q)
=
-U(\hat q)
<
\cdots
<
\mu_1(\hat q)
<
0,
\]
and their corresponding multiplicities by
$\alpha_1,\ldots,\alpha_k$. In particular,
\[
\alpha_l\geq2,
\qquad
k\geq l,
\]
and $k>l$ whenever $n>3$.

\begin{prop}\label{thm:Morse}
Let $\hat q$ be a planar balanced configuration in
$\{0\}\times\R^2$. Then the normal inertia indices of $\hat q$ vary
with $s>1$ as follows.

\begin{itemize}

\item[(i)]
If
\[
-\frac{\mu_i(\hat q)}{U(\hat q)}
<
s
<
-\frac{\mu_{i+1}(\hat q)}{U(\hat q)}
\]
for some $i\in\{l,\ldots,k-1\}$, then
\[
\iota_n^0(\hat q;s)=0,
\qquad
\iota_n^+(\hat q;s)
=
\sum_{j=1}^{i}\alpha_j,
\qquad
\iota_n^-(\hat q;s)
=
n-1-\sum_{j=1}^{i}\alpha_j.
\]

\item[(ii)]
If
\[
s
=
-\frac{\mu_i(\hat q)}{U(\hat q)}
\]
for some $i\in\{l+1,\ldots,k\}$, then
\[
\iota_n^0(\hat q;s)=\alpha_i,
\qquad
\iota_n^+(\hat q;s)
=
\sum_{j=1}^{i-1}\alpha_j,
\qquad
\iota_n^-(\hat q;s)
=
n-1-\sum_{j=1}^{i}\alpha_j.
\]

\item[(iii)]
If
\[
s>-\frac{\mu_k(\hat q)}{U(\hat q)},
\]
then
\[
\iota_n^0(\hat q;s)=0,
\qquad
\iota_n^+(\hat q;s)=n-1,
\qquad
\iota_n^-(\hat q;s)=0.
\]

\end{itemize}
\end{prop}

\begin{proof}
The normal block
\[
A_s(\hat q)
=
M^{-1}B(\hat q)+sU(\hat q)I_n
\]
has eigenvalues
\[
\mu_j(\hat q)+sU(\hat q),
\qquad
j=1,\ldots,k,
\]
with multiplicities $\alpha_j$. Hence its inertia indices are completely
determined by the signs of these shifted eigenvalues.
For $s>1$, all eigenvalues satisfying
$\mu_j(\hat q)\geq -U(\hat q)$ are positive. If
\[
-\frac{\mu_i(\hat q)}{U(\hat q)}
<
s
<
-\frac{\mu_{i+1}(\hat q)}{U(\hat q)},
\]
then
\[
\mu_j(\hat q)+sU(\hat q)>0
\quad\text{for }j\leq i,
\]
whereas
\[
\mu_j(\hat q)+sU(\hat q)<0
\quad\text{for }j\geq i+1.
\]
This proves~(i).
At
\[
s=-\frac{\mu_i(\hat q)}{U(\hat q)},
\]
the eigenvalue corresponding to $\mu_i(\hat q)$ vanishes with
multiplicity $\alpha_i$, which gives~(ii). Finally, if
\[
s>-\frac{\mu_k(\hat q)}{U(\hat q)},
\]
all the shifted eigenvalues are positive, proving~(iii).
\end{proof}

Thus, the normal Morse index changes precisely at the values
\[
s_i
\defeq
-\frac{\mu_i(\hat q)}{U(\hat q)},
\qquad
i=l+1,\ldots,k,
\]
and, as $s$ increases through $s_i$, it decreases by $\alpha_i$.
Equivalently,
\[
\iota_n^-(\hat q;s_i-\varepsilon)
-
\iota_n^-(\hat q;s_i+\varepsilon)
=
\alpha_i
\]
for $\varepsilon>0$ sufficiently small. In particular, the normal
Hessian $A_s(\hat q)$ is non-degenerate for every $s>1$ except at the
finitely many values $s_{l+1},\ldots,s_k$.


\section{Bifurcations from persistently degenerate branches of critical points}
\label{sec_2.2}

The aim of this section is to prove an abstract bifurcation theorem for branches of critical points that remain degenerate along the whole parameter interval. We begin by fixing the notation and recalling the
classical criterion in finite dimension, according to which a change of the Morse index
between two non-degenerate endpoints forces bifurcation. We then develop
an abstract framework inspired by the structure of the $S$-balanced
configuration problem, suitable for treating persistently degenerate
branches. Thus, let $M$ be a finite-dimensional smooth manifold, let $I=[a,b]$, and let
$F:I\times M\longrightarrow\R$
be a $\mathcal C^2$ function. We write
$F_\lambda \defeq F(\lambda,\cdot).$
A $\mathcal C^1$ curve $\gamma:I\longrightarrow M$
is called a \emph{branch of critical points} if
\[
dF_\lambda(\gamma(\lambda))=0, \quad \forall \lambda\in I.
\]

\begin{dfn}
A parameter $\lambda_\ast\in I$ is called a \emph{bifurcation instant}
from the branch $\gamma$ if there exist sequences
$\lambda_j\to\lambda_\ast$ and 
$x_j\to\gamma(\lambda_\ast)$, $x_j\neq \gamma(\lambda_j)$ for all $j$,
such that
\[
dF_{\lambda_j}(x_j)=0, \quad \forall j.\]
\end{dfn}

For each $\lambda\in I$, let
\[
H_\lambda
\defeq
\operatorname{Hess}F_\lambda(\gamma(\lambda))
\]
be the Hessian along the branch $\gamma$, and denote by
$\iota^-(H_\lambda)$ and $\iota^0(H_\lambda)$ its Morse index and
nullity, respectively. We say that the branch $\gamma$ is \emph{admissible} on
$I$ if
\[
\iota^0(H_a)=\iota^0(H_b)=0.
\]
The following classical bifurcation criterion will serve as a reference
point for the result proved below; see, for instance,
\cite[Theorem~4.4]{APF}.

\begin{thm}\label{thm:classical-bifurcation}
Let $(F_\lambda)_{\lambda\in I}$ be as above, and let $\gamma$ be an admissible branch of critical points. If
\[
\iota^-(H_a)\neq\iota^-(H_b),
\]
then there exists a bifurcation instant
$\lambda_\ast\in(a,b)$
from the branch $\gamma$. 
Assume in addition that $\ker H_\lambda = \{0\}$ for all but finitely many $\lambda\in I$, and set
\[
m\defeq
\max_{\lambda\in I}\dim\ker H_\lambda.
\]
Then there are at least
\[
\left\lceil
\frac{
\left|\iota^-(H_a)-\iota^-(H_b)\right|
}{m}
\right\rceil
\]
distinct bifurcation instants in $(a,b)$.
\end{thm}

\begin{rmk}
The finiteness assumption in
Theorem~\ref{thm:classical-bifurcation} is automatically satisfied if,
after a trivialization of the tangent bundle along $\gamma$, the family
of Hessians depends real-analytically on the parameter. Indeed,
$\det H_\lambda$ is then an analytic function which is not identically
zero, since the endpoint Hessians are non-degenerate.
\end{rmk}

Theorem~\ref{thm:classical-bifurcation} does not apply in general to the
$S$-balanced problem, since planar central configurations may remain genuinely degenerate even after symmetry reduction. An example is given by the equilateral triangle with an additional mass of value $m_\ast$ at the barycenter; see Section~\ref{section_6}. Since the planar block of the Hessian is independent of the
parameter $s$, such degeneracy persists
along the whole branch. This motivates the
criterion developed below.

We first recall the local topological invariant that will allow us to deal with
non-degeneracy in our argument. Let $f:N\to\R$ be a $\mathcal C^2$
function on a finite-dimensional manifold, and let $q$ be an isolated
critical point of $f$. Setting $c\defeq f(q)$ and
$f^c\defeq\{f(x)\leq c\}$, 
the \emph{$k$-th critical group} of $f$ at $q$, $k\in \N_0$, with coefficients in a
field $\mathbb K$, is
\[
C_k(f,q)
\defeq
H_k\big(
f^c\cap U,
(f^c\cap U)\setminus\{q\};
\mathbb K
\big),
\]
where $U$ is a small neighborhood of $q$ containing no
other critical points. The groups $C_k(f,q)$ are independent of the
choice of $U$. An isolated critical point $q$ of $f$ is called
\emph{homologically visible} if
\[
C_k(f,q)\neq 0
\]
for at least one $k$. We shall use the following standard continuation property of critical
groups; see e.g.~\cite[Theorem~5.2]{CorvellecHantoute}. In the present
finite-dimensional setting, the required local compactness condition is
automatic on $\overline U$.

\begin{prop}\label{prop:critical-groups-continuation}
Let $G:I\times N\to\R$ be a $\mathcal C^2$ function and let
$q\in N$ be a critical point of $G_\lambda$ for every $\lambda\in I$.
Assume that there exists a relatively compact open neighborhood $U$ of
$q$ such that $q$ is the only critical point of $G_\lambda$ in
$\overline U$ for every $\lambda\in I$. Then
\[
C_k(G_a,q)\cong C_k(G_b,q),
\qquad
\forall k\in\N_0.
\]
\end{prop}

We now introduce the geometric setting for our bifurcation criterion.
Let $\tau:M\to M$ be a smooth involution and denote by
$P\defeq\text{Fix}(\tau)$
its fixed-point set. Let
$F:I\times M\to\R$
be a $\mathcal C^2$ family of $\tau$-invariant functionals such that 
\[
F_\lambda|_P=f,
\qquad
\forall\,\lambda\in I,
\]
for some $\mathcal C^2$ function $f:P\to\R$, that is, such that the restriction of $F_\lambda$ to $P$ is
independent of the parameter. Choose a $\tau$-invariant Riemannian metric on $M$ and write
\[
TM|_P=TP\oplus NP.
\]
A critical point $q\in P$ of $f$ is automatically a critical point of
$F_\lambda$ for every $\lambda\in I$. Indeed, on $T_qP$ the differential
coincides with $df(q)$, while $\tau$-invariance forces it to vanish on
$N_qP$. Thus, $q$ determines a constant branch of critical points.
Moreover,
\[
\operatorname{Hess}F_\lambda(q)|_{T_qP}
=
\operatorname{Hess}f(q),
\]
and the mixed terms of the Hessian vanish with respect to the above
decomposition. We denote by
\[
A_\lambda
\defeq
\operatorname{Hess}F_\lambda(q)|_{N_qP}
\]
the \emph{normal Hessian} at $q$, and by
$\iota^-(A_\lambda)$
its Morse index.

\begin{thm}\label{thm:persistent-bifurcation}
Let $F:I\times M\to\R$ be a $\mathcal C^2$ family of functionals invariant by an involution $\tau:M\to M$, let
$P=\mathrm{Fix}(\tau)$, and assume that the restriction 
$F_\lambda|_P=:f$ is independent of $\lambda \in I$. Let $q\in P$ be an isolated homologically visible critical point of
$f$. Suppose that the normal Hessians $A_a$ and $A_b$ are non-degenerate and
that
\[
\iota^-(A_a)\neq\iota^-(A_b).
\]
Then there exists a bifurcation instant
$\lambda_\ast\in(a,b)$
from the constant branch $
\gamma(\lambda)\equiv q$.
\end{thm}

\begin{rmk}\label{rmk:extension}
The invariance under the involution $\tau$ in
Theorem~\ref{thm:persistent-bifurcation} is a natural geometric condition
ensuring the analytic structure used in the proof. More generally, one
may replace it by an embedded submanifold $P\subset M$ such that
$F_\lambda|_P=f$
is independent of $\lambda$ and, in a neighborhood of $q$,
\[
dF_\lambda(x)|_{N_xP}=0,
\qquad
x\in P,\ \lambda\in I.
\]
Equivalently, in coordinates $(u,v)$ adapted to $P$, these conditions read
\[
F_\lambda(u,0)=f(u),
\qquad
D_vF_\lambda(u,0)=0.
\]
In particular, the mixed second derivatives vanish along $P$, and the
Morse lemma with parameters can be applied in the normal directions while
keeping $f$ as the reduced function. We refrain from stating the theorem in this more general form and instead
use the formulation in terms of an involution, since it is the natural
one arising from the $S$-balanced configuration problem.
\end{rmk}

\begin{rmk}
Neither homological visibility nor the normal decoupling condition in Remark~\ref{rmk:extension} can
in general be dropped. Homological visibility is necessary, as
shown by
\[
F_\lambda(x,y)
=
\frac{\lambda}{2}x^2+\frac14x^4+x^2y+\frac13y^3,
\qquad
\lambda\in[-1,1].
\]
This family is invariant under $(x,y)\mapsto(-x,y)$, the normal Morse
index changes along the branch $(0,0)$, but the restriction to the fixed
point set is $f(y)=y^3/3$, whose isolated critical point at the origin is
homologically invisible, and no bifurcation occurs.
The normal decoupling condition is also
essential. Indeed, consider
\[
F_\lambda(x,y)
=
\frac12y^2+2xy+\frac{\lambda}{2}x^2,
\qquad
\lambda\in[-1,1],
\]
with $P=\{x=0\}$. Then
$F_\lambda|_P=f(y)=y^2/2$,
so that the origin is an isolated homologically visible critical point
of $f$, and the normal Morse index changes between the two endpoints.
Nevertheless, the origin is the only critical point of $F_\lambda$ for
every $\lambda\in[-1,1]$. In this case the normal decoupling condition
fails, since $
\partial_xF_\lambda(0,y)=2y.$
\end{rmk}

\begin{rmk}
We expect Theorem~\ref{thm:persistent-bifurcation} to admit an
infinite-dimensional extension in which the difference of the normal
Morse indices is replaced by the spectral flow of the normal Hessians.
The natural setting would be a Hilbert manifold for which, after a
trivialization of the normal bundle, the normal Hessians define a
norm-continuous path
$A_\lambda\in\Phi_{\mathrm{sa}}(\mathcal H)$
of bounded self-adjoint Fredholm operators with invertible endpoints.
The corresponding condition would then be
\[
\operatorname{sf}(A_\lambda,\lambda\in[a,b])\neq0.
\]
An analogous formulation should also be possible for suitable
gap-continuous paths of unbounded self-adjoint Fredholm operators. Such an extension will not be pursued here.
\end{rmk}

\begin{proof}
Assume by contradiction that there is no bifurcation instant in $(a,b)$.
We first claim that there exists a relatively compact open neighborhood
$U$ of $q$ such that $q$ is the only critical point of $F_\lambda$ in
$\overline U$ for every $\lambda\in I$. To see this, consider first the endpoint $a$. Using the $\tau$-invariant metric and the equivariant tubular
neighborhood theorem
(see~\cite[Chapter~VI, Theorems~2.1--2.2]{Bredon}), choose local
coordinates $(u,v)$ centered at $q$ such that
$P=\{v=0\}$ and $\tau(u,v)=(u,-v).$
Hence, by $\tau$-invariance,
$F_\lambda(u,v)=F_\lambda(u,-v)$,
and therefore
\[
F_\lambda(u,0)=f(u),
\qquad
D_vF_\lambda(u,0)=0.
\]
Since $A_a=D^2_{vv}F_a(0,0)$ is non-degenerate, the implicit function
theorem implies that, for $\lambda$ close to $a$ and $(u,v)$ close to
$(0,0)$, the equation
\[
D_vF_\lambda(u,v)=0
\]
has the unique solution $v=0$. Hence every critical point of
$F_\lambda$ sufficiently close to $q$ lies in $P$ and is therefore a
critical point of $f$. Since $q$ is isolated as a critical point of
$f$, it follows that $q$ is uniformly isolated as a critical point of
$F_\lambda$ for $\lambda$ sufficiently close to $a$. The same argument
applies near $b$.
If the claim were false, we could find sequences $\lambda_j\in I$ and
$x_j\neq q$ such that
\[
dF_{\lambda_j}(x_j)=0,
\qquad
x_j\to q.
\]
After passing to a subsequence, $\lambda_j\to\lambda_\ast\in I$. The
uniform isolation near the endpoints implies that
$\lambda_\ast\in(a,b)$. Hence $\lambda_\ast$ would be a bifurcation
instant, a contradiction. This proves the claim.

Proposition~\ref{prop:critical-groups-continuation} now implies that
\[
C_k(F_a,q)\cong C_k(F_b,q),
\qquad
\forall k\in\N_0.
\]
We now compute the critical groups at the endpoints. Since $A_a$ is
non-degenerate, we may apply the parameterized Morse lemma for
$\mathcal C^2$ functionals
(see, e.g.,~\cite[Remark~A.8]{LuParameterized}) to 
$
v\mapsto F_a(u,v)-f(u)$,
with $u$ as parameter. Hence there exists a local fiber-preserving
homeomorphism
$(u,w)\mapsto (u,\psi(u,w))$
such that
\[
F_a(u,\psi(u,w^-,w^+))
=
f(u)-\|w^-\|^2+\|w^+\|^2,
\]
where
\[
w^-\in\R^{\iota^-(A_a)},
\qquad
w^+\in\R^{d-\iota^-(A_a)},
\qquad
d=\dim N_qP.
\]
Consequently, by the shifting property of critical groups
(see, e.g.,~\cite[Theorem~10]{PereraMorse}),
\[
C_k(F_a,q)
\cong
C_{k-\iota^-(A_a)}(f,q).
\]
Similarly,
\[
C_k(F_b,q)
\cong
C_{k-\iota^-(A_b)}(f,q).
\]
Here and in what follows, we adopt the convention
$C_j(f,q)=0$ for $j<0$.
Suppose, without loss of generality, that
\[
\iota^-(A_a)<\iota^-(A_b).
\]
Since $q$ is homologically visible, there exists a smallest
$r\in\N_0$ such that
$C_r(f,q)\neq0.$
Taking
$k=r+\iota^-(A_a),$
we obtain
\[
C_k(F_a,q)
\cong
C_r(f,q)\neq0,
\]
whereas
\[
C_k(F_b,q)
\cong
C_{r+\iota^-(A_a)-\iota^-(A_b)}(f,q)=0,
\]
since
$r+\iota^-(A_a)-\iota^-(A_b)<r$, a contradiction.
Hence there exists a bifurcation instant
$\lambda_\ast\in(a,b)$ from the constant branch.
\end{proof}


\section{Bifurcations of planar balanced configurations}
\label{sec_2.3}

We now combine the normal index computation of Section~\ref{sec_2.1} with
Theorem~\ref{thm:persistent-bifurcation} to study bifurcations from the
branches generated by planar central configurations.
Let $\hat q$ be a non-collinear planar central configuration in
$\{0\}\times\R^2$, and let $q=[\hat q]\in P$ denote its class modulo the
$SO(2)$-action. Recall that
$f:=U|_P$
is independent of the parameter $s$.

\begin{thm}\label{thm:planar-bifurcation}
Let $\hat q$ be a non-collinear planar central configuration and assume
that its class $q=[\hat q]\in P$ is an isolated homologically visible
critical point of
$f:=U|_P$.
Let
\[
\mu_k(\hat q)<\cdots<\mu_l(\hat q)=-U(\hat q)
<\cdots<\mu_1(\hat q)<0
\]
be the distinct nonzero eigenvalues introduced in Section~\ref{sec_2.1}. Then, 
the parameter
\[
s_i=-\frac{\mu_i(\hat q)}{U(\hat q)}, \quad \forall i=l+1,\ldots,k,
\]
is a bifurcation instant from the planar branch generated by $\hat q$.
In particular, there are at least $k-l$ distinct bifurcation instants
for $s>1$. Moreover, the bifurcating critical points can be chosen outside $P$ and
therefore correspond to non-planar balanced configurations.
\end{thm}

\begin{proof}
Using the diffeomorphisms introduced in Section~\ref{sec_2.1}, we identify all the
quotient spaces $\overline{\mathcal S}_s$ with
$\overline{\mathcal S}_1$. More precisely, let
$\overline\Phi_s:
\overline{\mathcal S}_1
\longrightarrow
\overline{\mathcal S}_s$
be the diffeomorphism induced by $\Phi_s$, and define
\[
F_s
\defeq
U|_{\overline{\mathcal S}_s}\circ\overline\Phi_s, 
\]
so that $(F_s)_{s>1}$ is a smooth family of functionals on the fixed manifold $\overline{\mathcal S}_1$.
Since $\overline\Phi_s$ commutes with the involution $\tau$ and restricts
to the identity on $P$, each $F_s$ is $\tau$-invariant and
$F_s|_P=f$
for every $s>1$. Moreover, $q=[\hat q]$ is a critical point of $F_s$ for
every $s>1$, and hence defines a constant branch of critical points. The normal Hessian of $F_s$ at $q$ is congruent to the normal block
$A_s(\hat q)$ computed in Section~\ref{sec_2.1}. Notice that passing to the
$SO(2)$-quotient does not affect the normal directions, since the
infinitesimal $SO(2)$-action is tangent to $P$. Therefore, the normal
inertia indices are precisely those described in
Proposition~\ref{thm:Morse}. Fix $i\in\{l+1,\ldots,k\}$. Since the values $s_j$ are distinct, we can
choose $\varepsilon>0$ such that
\[
1<s_i-\varepsilon<s_i<s_i+\varepsilon
\]
and $s_i$ is the only value $s_j$ contained in
$[s_i-\varepsilon,s_i+\varepsilon]$.
By Proposition~\ref{thm:Morse}, the normal Hessians at the two
endpoints are non-degenerate and
\[
\iota^-\big(A_{s_i-\varepsilon}(\hat q)\big)
-
\iota^-\big(A_{s_i+\varepsilon}(\hat q)\big)
=
\alpha_i>0.
\]
Hence Theorem~\ref{thm:persistent-bifurcation} yields a bifurcation instant $s_\ast \in (s_i-\epsilon,s_i+\epsilon)$, and one readily sees that $s_\ast=s_i$. Since the eigenvalues $\mu_i(\hat q)$ are distinct,
the corresponding $s_i$ are distinct, and hence there are at least $k-l$ bifurcation instants.
Finally, since $q$ is isolated as a critical point of $f=F_s|_P$,
critical points bifurcating from $q$ cannot lie in $P$ for $s$ and the
configuration sufficiently close to $(s_i,q)$. \end{proof}


\section{Numerical simulations for planar balanced configurations}
\label{section_6}

Theorem~\ref{thm:planar-bifurcation} identifies the bifurcation instants
from planar central configurations satisfying its hypotheses, indicating
where to look numerically for the emerging non-trivial branches. In this section, we use numerical simulations to compute the emerging branches of balanced configurations.
The simulations were obtained using a continuation method based on a program originally developed by Fenucci in~\cite{APF} for collinear central configurations in the three-body problem and adapted here to the planar setting. 

\medskip

The continuation method~\cite{Allgower1990} is a numerical technique for tracing families of solutions of parametric nonlinear systems of the form 
\[
F(q, s) = 0,
\]
where $q \in \mathbb{R}^d$ is the configuration vector and $s \in \mathbb{R}$ denotes the parameter. Starting from a known solution $(q_i, s_i)$, the goal is to compute new solutions for different values of the parameter $s$. At each step, we displace $(q, s)$ by a step $\delta$ and solve the system
\[
\begin{cases}
F(q, s) = 0, \\
\lvert (q, s) - (q_i, s_i) \rvert^2 - \delta^2 = 0,
\end{cases}
\]
using a damped Newton’s method, whose Jacobian has the form
\[
\begin{pmatrix}
\dfrac{\partial F}{\partial q} & \dfrac{\partial F}{\partial s} \\
2(q - q_i) & 2(s - s_i)
\end{pmatrix}.
\]
An initial guess is obtained by starting from the known solution $(q_i, s_i)$ and taking  a tangent displacement along the curve of solutions.

In our case, the system is given by the balanced configuration equation 
\[
F(q, s) = M^{-1} \nabla U(q) + U(q) \hat{S}(s) q.
\]
Starting from a planar central configuration $\hat{q}$, which produces a trivial branch of solutions, we compute the parameters $\tilde{s}$ at which the configuration becomes degenerate, and hence bifurcations arise. We then displace $\tilde{s}$ by setting $s = \tilde{s} +\varepsilon \Delta s$, where $\varepsilon\leq 1$ is the damping parameter, and compute a new solution using Newton’s method. An initial guess is constructed by displacing $\hat{q}$ along the direction of the kernel of $\frac{\partial F}{\partial q}$. These two solutions are then used to initialize the continuation method.

\subsection{Examples for $n = 4$}

We present two numerical examples for the four-body problem. As starting planar central configurations, we consider a square with equal masses and an equilateral triangle with equal masses and an additional mass at the barycenter.

The first example, shown in Figure~\ref{fig_1}, investigates the bifurcations originating from the central configuration with four equal unit masses at the vertices of a square. The trivial branch generated by this configuration undergoes a bifurcation at $\bar{s} \approx 1.4$, from which a non-trivial branch of spatial balanced configurations emerges. Along this branch, the parameter $s$ decreases, while the masses move to form the vertices of a tetrahedron. At $s = 1$, the configuration becomes a regular tetrahedron, corresponding to a spatial central configuration for the four-body problem. The solutions along the trivial branch of planar configurations are saddle points of $U|_\s$ for $s < \bar{s}$ and local minima for $s>\bar{s}$, while along the non trivial branch we have local minima. 
We observe, therefore, a connection between the square central configuration and the regular tetrahedron through a branch of spatial balanced configuration which are local minima of the potential.

\begin{figure}[h!]
    \centering
    \includegraphics[width=.5\textwidth]{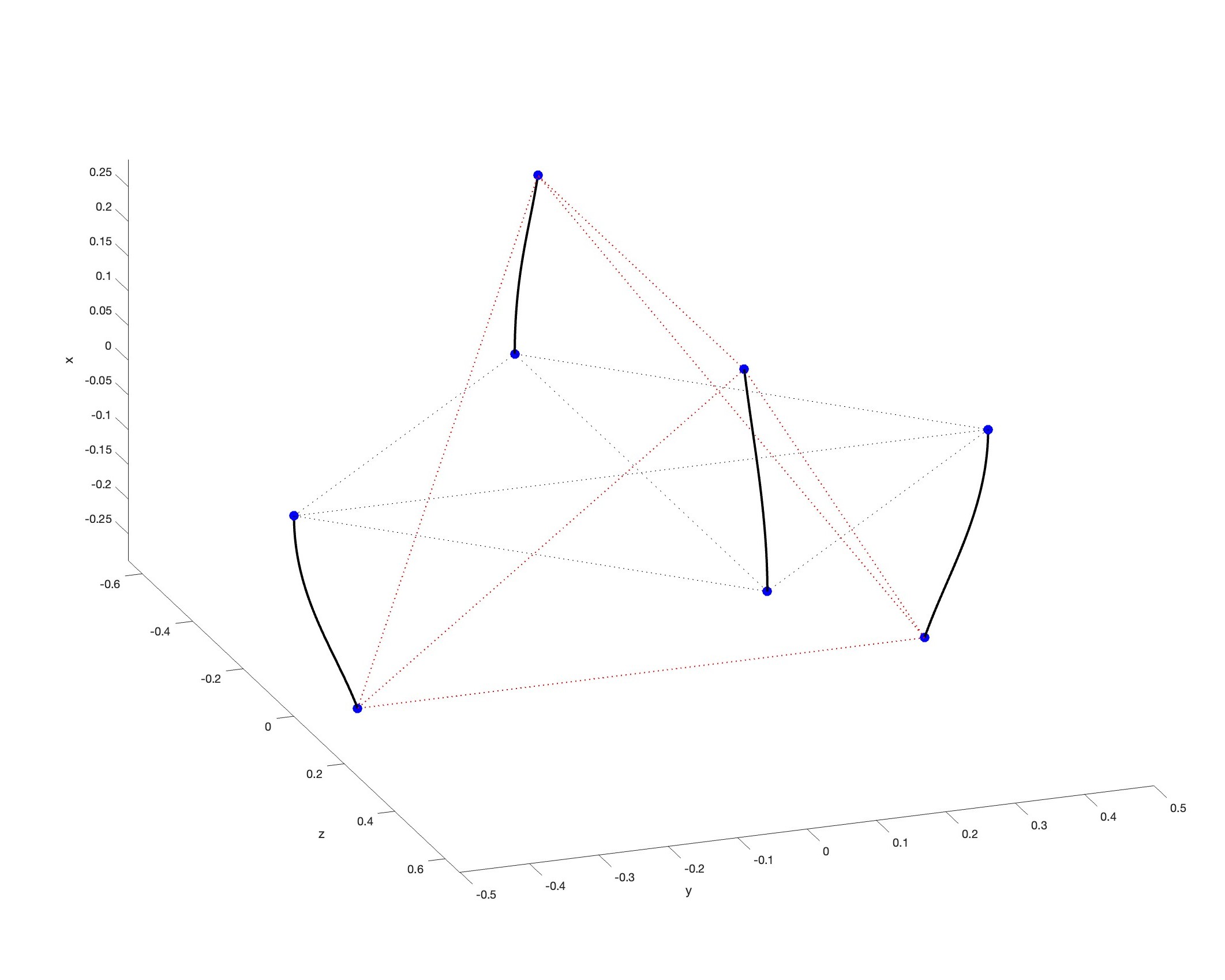}
    \caption{Non-trivial branch of balanced configurations bifurcating from the square configuration with four equal masses. 
    As $s$ decreases from $\bar{s} \approx 1.4$, the square deforms into a tetrahedron, which becomes regular at $s=1$. 
    All configurations along this branch are minima of $U|_\s$.}
    \label{fig_1}
\end{figure}

\noindent In the second example, we examine an equilateral triangle with three unit masses at the vertices and an additional mass $m_4$ at the barycenter. The behavior of this configuration is a lot richer. In particular, this planar central configuration admits a critical value $m^*$ of $m_4$:
\[
m^*= \frac{2 + 3 \sqrt{3}}{18 - 5 \sqrt{3}},
\]
at which the configuration becomes degenerate, transitioning from a local minimum to a saddle. This value corresponds to a bifurcation point for the mass parameter $m_4$, see~\cite{sim_relative_1978}. 
In our analysis, we fix the value of $m_4$ and instead study bifurcations associated with the parameter $s$. 

\begin{rmk}
The hypotheses of Theorem~\ref{thm:planar-bifurcation} are satisfied  for $m_4=m_\ast$. Indeed, after symmetry reduction the planar
Hessian has a two-dimensional kernel, and the Lyapunov--Schmidt
expansion computed in~\cite{SantosLeandro2026} yields, in the
variational formulation, a reduced potential whose leading term is
$
c\,(u^3-3uv^2),\ c\neq0.$
Thus the reduced critical point is an isolated monkey-saddle
singularity and has non-trivial local critical homology. By the
shifting property of critical groups, the same holds for the original
critical point. 
\end{rmk}

Along the trivial branch, we find a bifurcation point $\bar{s}$, at which a non-trivial branch bifurcates. As in the previous example, the parameter $s$ decreases, and the masses form a tetrahedron which converges to a regular one for $s \to 1$. The nature of the solutions depend on the value of $m_4$. 

\begin{figure}[h!]
    \centering
    \includegraphics[width=.5\textwidth]{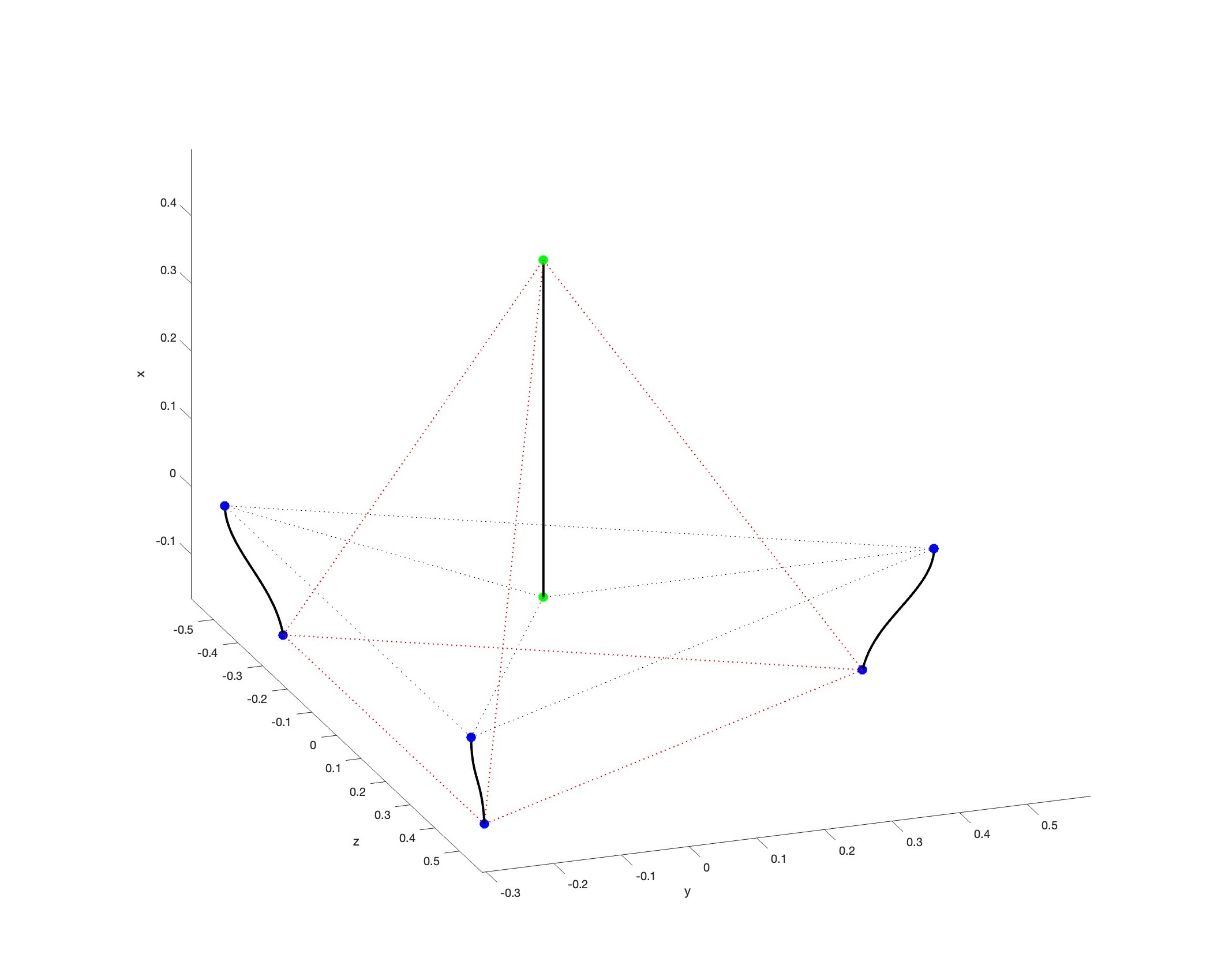}
    \caption{Non-trivial branch bifurcating from the equilateral triangle with three unit masses and a unit central mass. The branch originates at $\bar{s} \approx 2.5$.}
    \label{fig_2}
\end{figure}

\noindent For $m_4 < m^*$, the initial planar configuration is a saddle for $s < \bar{s}$ and a local minimum for $s > \bar{s}$, while the non-trivial branch consists of local minima. As $m_4$ approaches $m^*$, the solutions along the non-trivial branch transition from minima to saddle points, and then back to minima near $s = 1$. For $m_4 > m^*$, the initial configuration is always a saddle point, with its Morse index decreasing by one at the bifurcation point $\bar{s}$. The non-trivial branch consists of saddle points, which transition to minima as $s$ approaches 1. Figure~\ref{fig_2} shows the bifurcation happening at $\bar{s} \approx 2.5$ when $m_4=1 > m^*$.


\subsection{Examples for $n = 5$}
For the five-body problem, we consider the planar central configuration consisting of four unit masses located at the vertices of a square and an additional mass $m_5$ at the barycenter. As in the $3+1$-body problem discussed above, there exists a critical value $m^*$ of $m_5$ at which the configuration changes from a local minimum to a saddle point.
For any choice of $m_5$, we observe two bifurcation points $s_1 > s_2$. At the first bifurcation point, a nontrivial branch emerges along which the parameter $s$ decreases, and the configuration approaches, as $s \to 1$, a regular square pyramid. For $m_5 < m^*$, this branch initially consists of local minima and transitions to saddle points as $s$ approaches $1$; for $m_5 > m^*$, the branch consists entirely of saddle points.
Along the branch originating from $s_2$, the parameter $s$ decreases and the configuration converges to a regular tetrahedron with the mass $m_5$ located at the barycenter. The two types of behavior for $m_5 = 1$ are illustrated in Figures~\ref{fig:figure1} and~\ref{fig:figure2}. 

\begin{figure}[h!]
    \centering
    \begin{minipage}{0.49\textwidth}
        \centering
        \includegraphics[width=\textwidth]{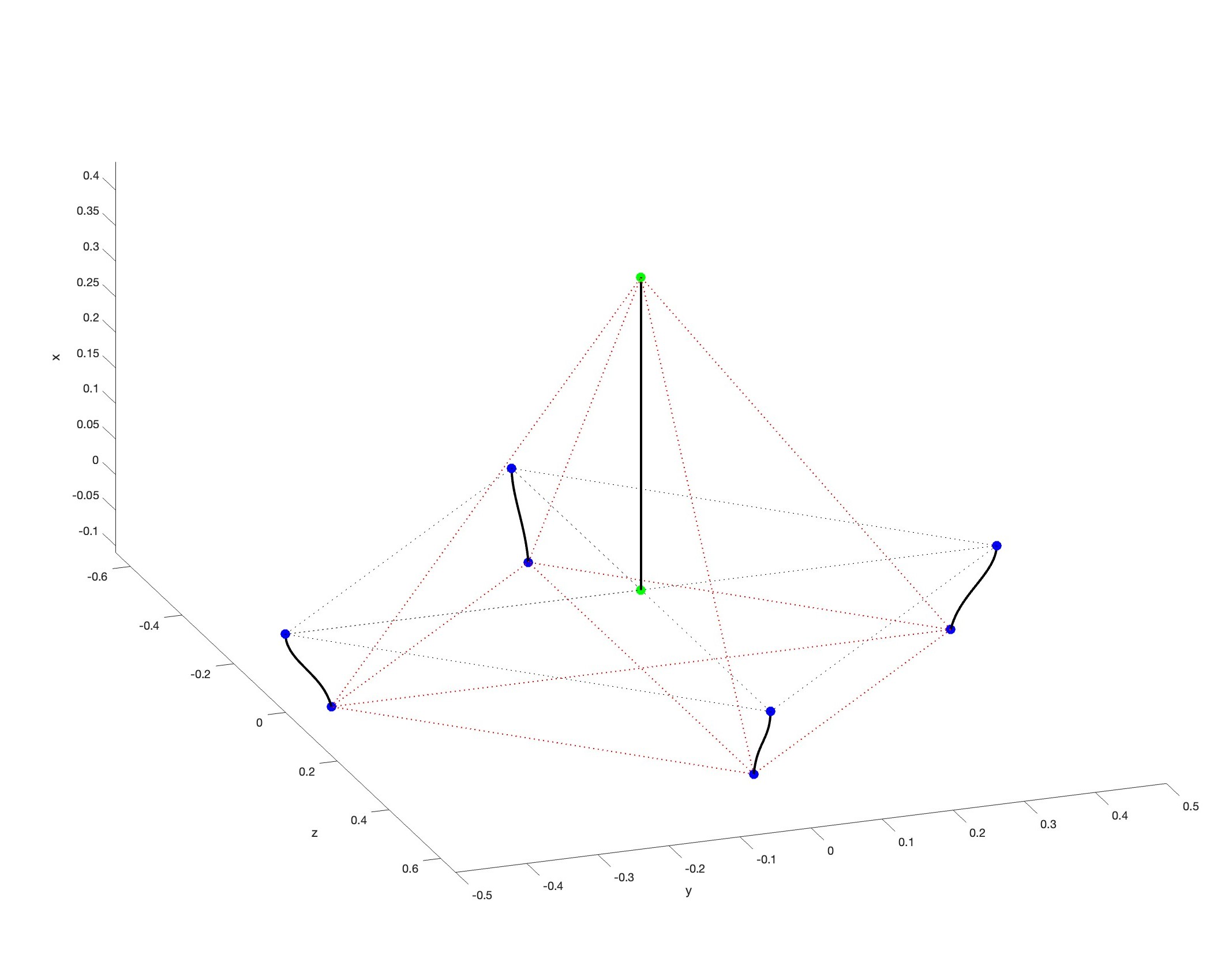}
        
        \caption{Non-trivial branch for five equal masses originating at $s_1 \approx 2.5$. 
        As $s$ decreases, the configuration evolves from a square with a central mass into a regular square pyramid at $s=1$.}
        \label{fig:figure1}
    \end{minipage} \hfill
    \begin{minipage}{0.49\textwidth}
        \centering
        \includegraphics[width=\textwidth]{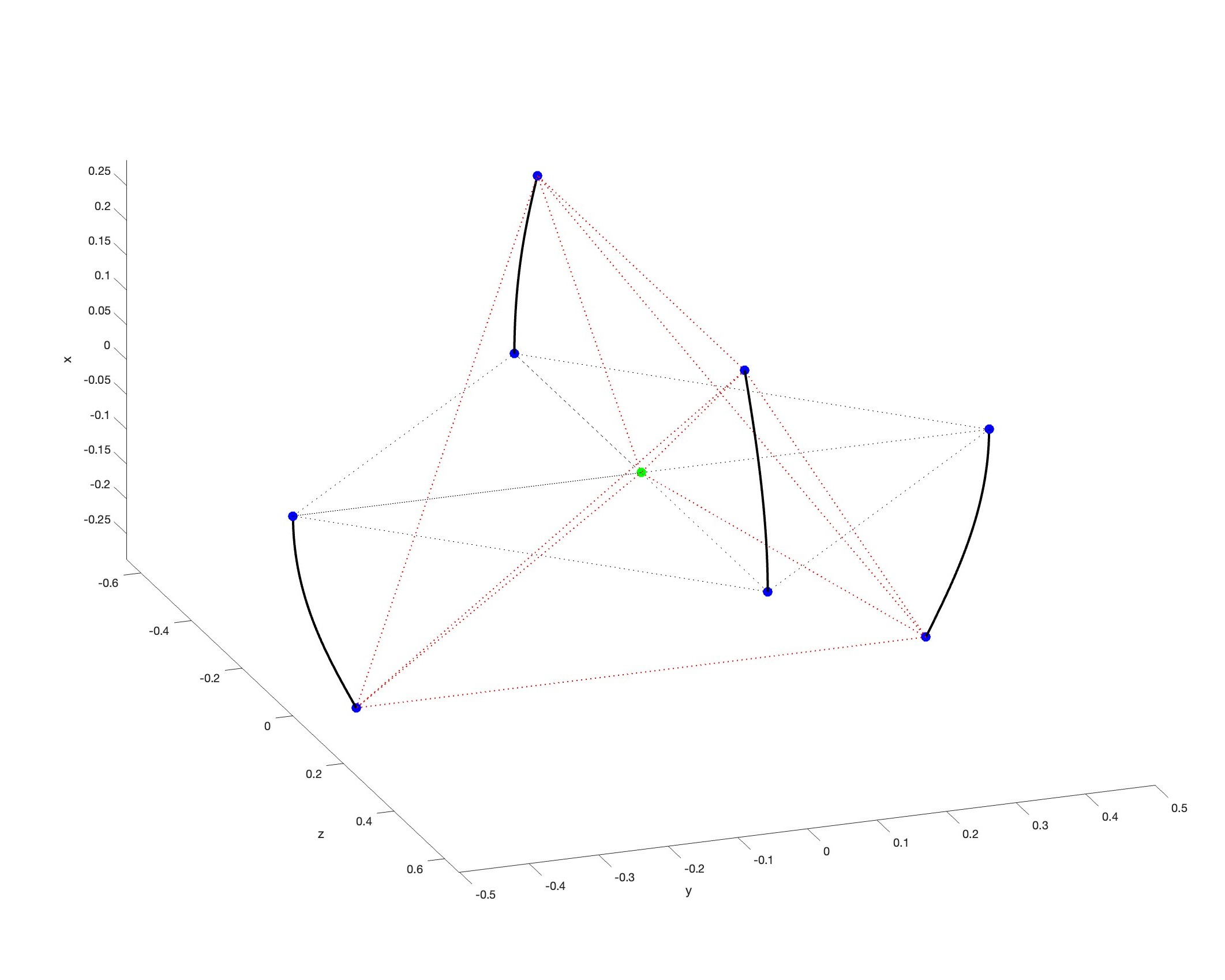} 
        \vspace{-0.7cm}
        \caption{Non-trivial branch for five equal masses originating at $s_2 \approx 1.2$. 
        As $s$ decreases, the configuration becomes a tetrahedron with the fifth mass at the barycenter.}
        \label{fig:figure2}
    \end{minipage}
\end{figure}

\section{Numerical simulations for collinear balanced configurations}
\label{section_7}

We conclude by presenting some numerical examples of bifurcations of collinear central configurations for the four-body problem. In~\cite{APF}, numerical simulations were carried out for the three-body case; here, we extend that work to four bodies.
Recall that for this problem we work in $\R^2$ with
\[
S = \mathrm{diag}(s,1).
\]
Collinear central configurations on $\{0\} \times \R$ are solutions of the balanced configuration equation for all $s > 1$, forming a trivial branch of solutions. For each choice of mass ordering along the $y$-axis, there are at least $n - 2$ bifurcation instants. In particular, for $n=4$ there are always at least two.

\medskip

As a first example, we examine the collinear central configuration with all four masses equal to 1. The results are depicted in Figures~\ref{fig_3} and~\ref{fig_4}; the dashed lines represent the trivial branch generated by the collinear central configuration. We observe two bifurcation points, $s_1 \approx 4.1$ and $s_2 \approx 2.4$. At each point, two non-trivial branches of solutions detach, symmetric with respect to the $y$-axis. Along the branch originating from $s_1$ (see Figure~\ref{fig_3}), the parameter $s$ decreases until reaching a turning point, at which the masses form the vertices of a rhombus. The parameter then increases again, returning to $s_1$, with the masses realign along a line, but with the two central masses exchanged. During this process, the two central masses appear to rotate around a common point, while the
outer masses oscillate. All these configurations are local minima of $U|_{\s}$. At the turning
point, two additional branches bifurcate (the red curves in the figure). Along one branch, the parameter $s$ increases reaching a limit configuration for $s \to + \infty$; this branch consists of saddle points. Along the other branch, $s$ decreases, converging to a square central
configuration with its vertices on the coordinate axes for $s \to 1$. These solution are all local minima.

\begin{figure}[h!]
    \centering
    \begin{minipage}{0.49\textwidth}
        \centering
        \includegraphics[width=0.55\textwidth]{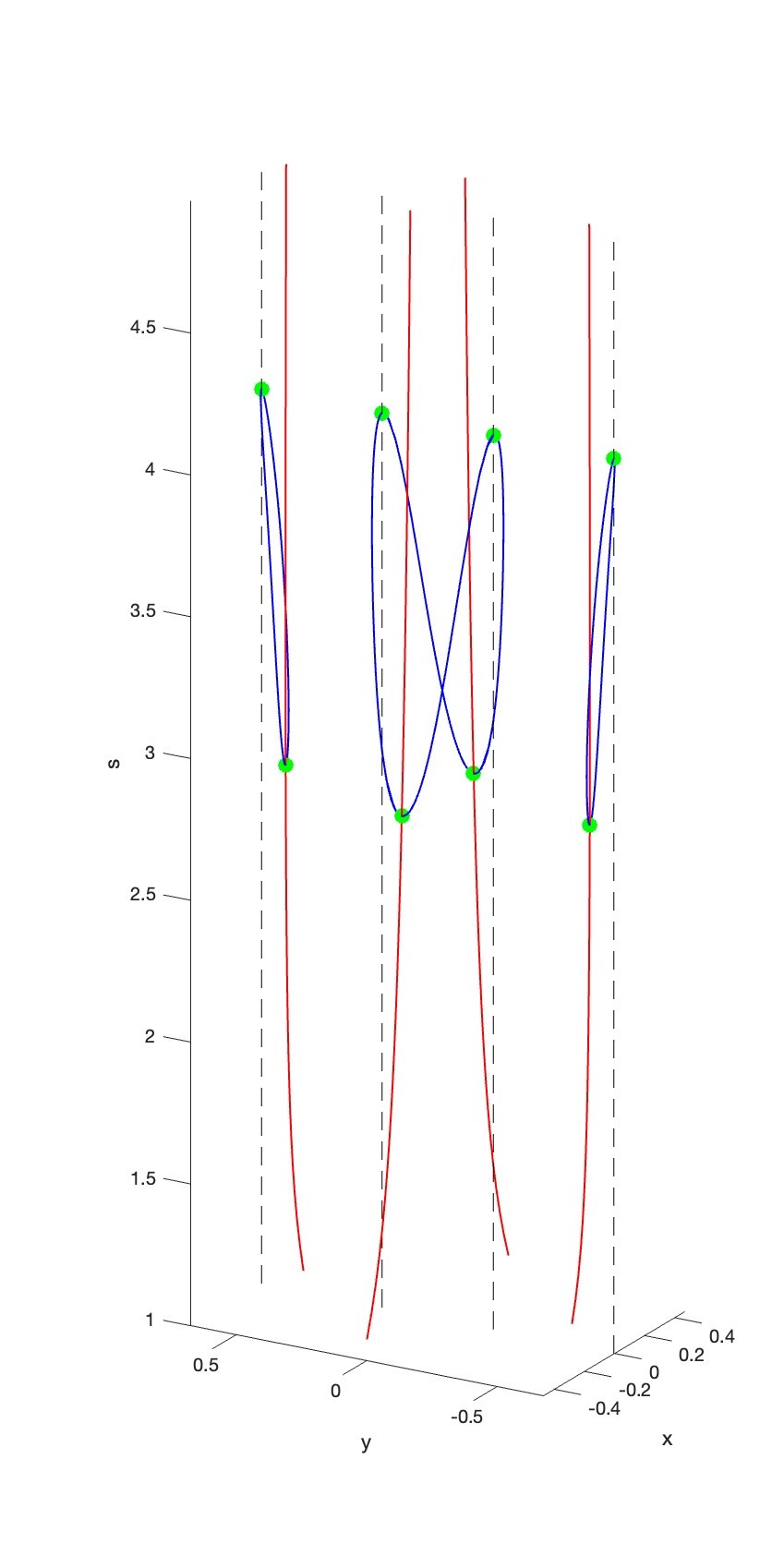} 
        \caption{Non-trivial branches for masses $[1,1,1,1]$ originating at $s_1 \approx 4.1$. 
        The two central masses rotate around each other while the outer ones oscillate. 
        At the turning point, two new branches detach, one for increasing $s$, and one for decreasing $s$ reaching a square central configuration for $s \to 1$.}
        \label{fig_3}
    \end{minipage} \hfill
    \begin{minipage}{0.49\textwidth}
        \centering
        \includegraphics[width=0.55\textwidth]{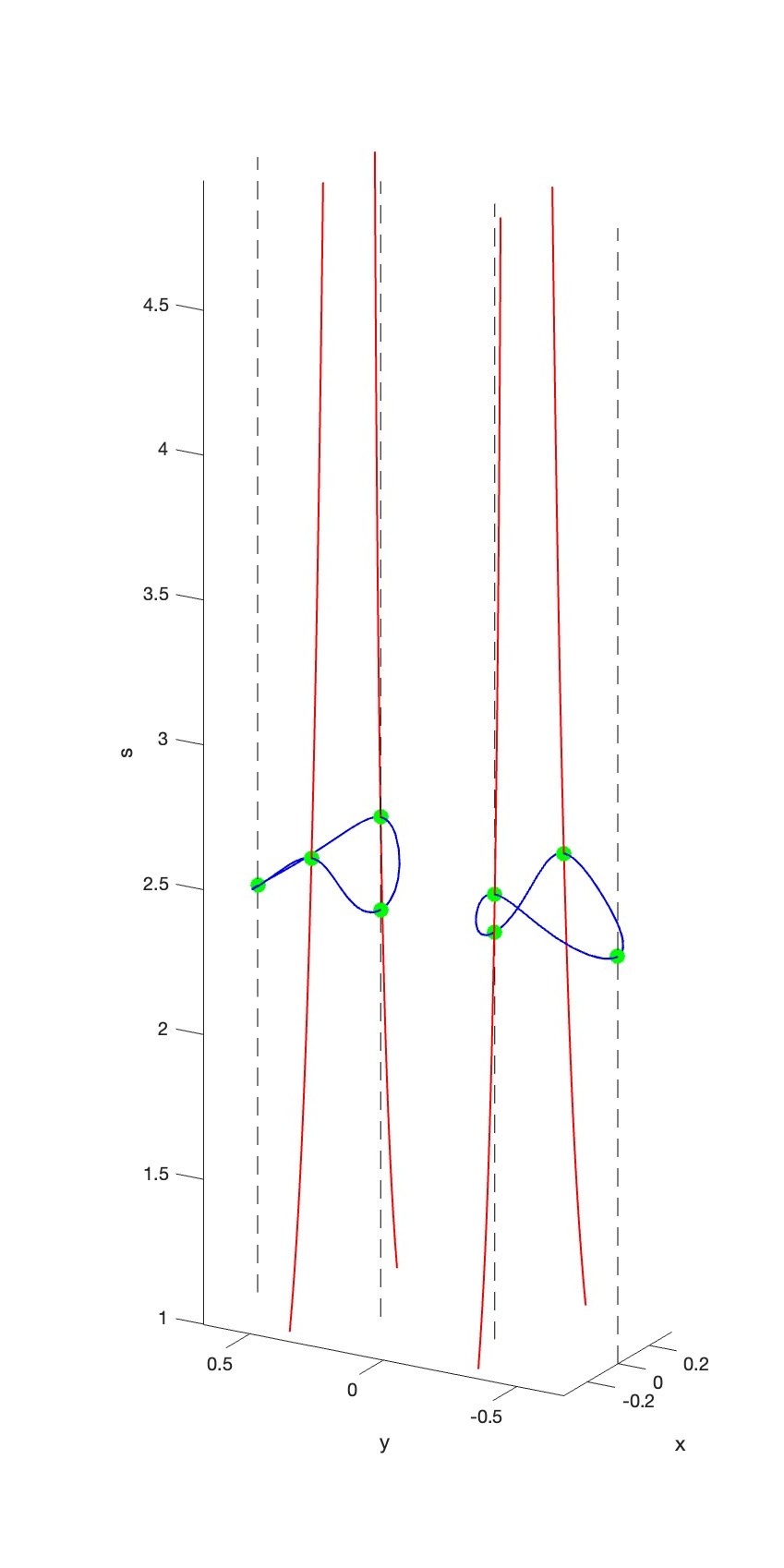} 
        \caption{Non-trivial branch for masses $[1,1,1,1]$ originating at $s_2 \approx 2.4$. The first two masses and the last two rotate around each other. At the turning point, two new branches detach, one for increasing $s$, and one for decreasing $s$ reaching a square central configuration for $s \to 1$.}
        \label{fig_4}
    \end{minipage}
\end{figure}

\noindent Along the branch originating from $s_2$ (see Figure~\ref{fig_4}), the parameter $s$ increases,
reaches a turning point, and then decreases until the masses realign along a line, with the first
and second masses exchanged, as well as the third and fourth masses. All configurations along this
branch are saddle points of $U|_{\s}$. At the turning point, two further branches bifurcate: along the branch corresponding to increasing $s$, the solutions remain saddle points with the same Morse index as the main branch, whereas along the branch corresponding to decreasing $s$, the solutions are still saddle points but with Morse index decreased by one. Also in this case, for $s \to 1$, we reach a square central configuration with sides parallel to the coordinate axes.

In the next three examples, we assume the first two masses to be equal to 1 (represented by the green dots in the figures), while the last two masses are equal to $\alpha < 1$ (represented by the magenta dots). As before, the dashed lines represent the trivial branch of collinear central configurations. In all cases, we observe two bifurcation points, each giving rise to two non-trivial branches that are symmetric with respect to the $y$-axis. For clarity, Figures~\ref{fig_6},~\ref{fig_7}, and~\ref{fig_9} show only one of these branches. We analyze how the behavior changes as the mass parameter $\alpha$ varies.
We consider first $\alpha = 0.2$, for which we find two bifurcation points: $s_1 \approx 4.9$ and $s_2 \approx 1.9$. 
Along the non-trivial branches emerging from the first bifurcation point (see Figure~\ref{fig_5}), the parameter $s$ decreases, reaches a turning point, and then increases again up to $s_1$, where the configuration realigns with the two smaller masses exchanged. 

\begin{figure}[h!]
    \centering
    \begin{minipage}{0.49\textwidth}
        \centering
        \includegraphics[width=0.55\textwidth]{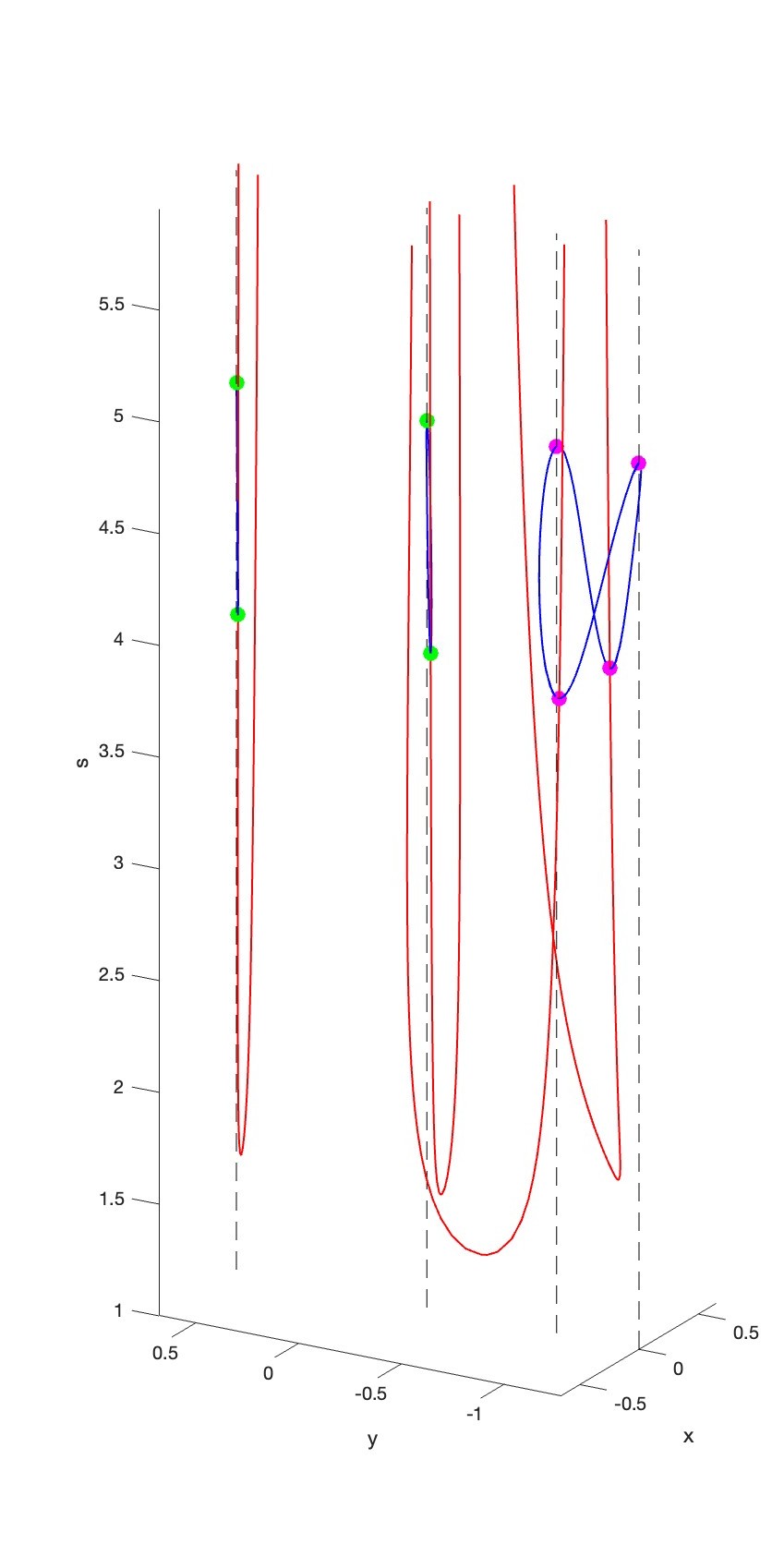} 
        \caption{Non-trivial branches for masses $[0.2,0.2,1,1]$ originating at $s_1 \approx 4.9$. 
        The smaller masses rotate around each other while the bigger masses oscillate. 
        At the first turning point, two new branches detach.}
        \label{fig_5}
    \end{minipage} \hfill
    \begin{minipage}{0.49\textwidth}
        \centering
        \includegraphics[width=0.55\textwidth]{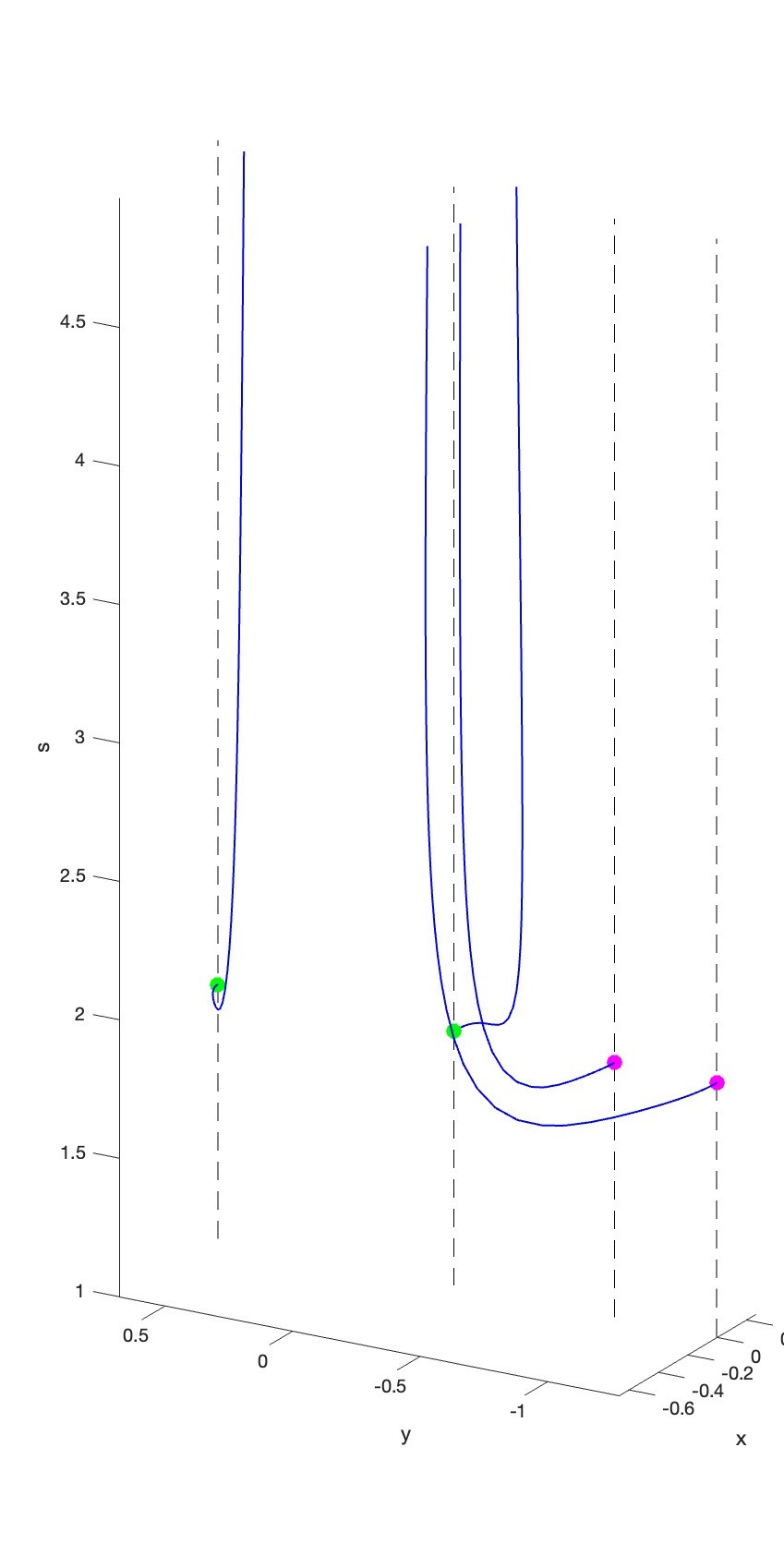} 
        \caption{Non-trivial branches for masses $[0.2,0.2,1,1]$ originating at $s_2 \approx 1.9$. 
        Along these branches the parameter $s$ increases, and the configuration tends toward a limit one as $s \to +\infty$.}
        \label{fig_6}
    \end{minipage}
\end{figure}

\noindent This corresponds to the fact that the smaller masses rotate about a common center, while the unit masses exhibit oscillatory motion.
All solutions along this branch correspond to local minima. At the turning point, two additional branches appear. The branch for increasing $s$ consists of saddle points. The branch for decreasing $s$ initially consists of minima. It encounters a bifurcation point at $s \approx 3.9$, from which two additional branches bifurcate (not shown in the picture) and the solutions along the primary branch become saddle points. It then reaches a turning point at $s \approx 1.5$, where there is another jump in the Morse index (though the solutions remain saddle points), but no additional branches are found.
This branch then approaches a limiting configuration as $s \to +\infty$.
Along the branches originating from $s_2$, the parameter $s$ decreases until it reaches a turning point, after which it increases, approaching a limiting configuration as $s \to +\infty$ (see Figure~\ref{fig_6}). All solutions along this branch are saddle points, although the Morse index differs before and after the turning point.

Next, we choose $\alpha = 0.5$, see Figure~\ref{fig_7} and Figure~\ref{fig_8}. Along the branches originating from the first bifurcation instant $s_1 \approx 4.3$, the parameter $s$ initially decreases until a turning point, after which it increases, while the configuration approaches a limit configuration for $s \to + \infty$.

\begin{figure}[h!]
    \centering
    \vspace{-1cm}
    \begin{minipage}{0.49\textwidth}
        \centering
        \includegraphics[width=0.50\textwidth]{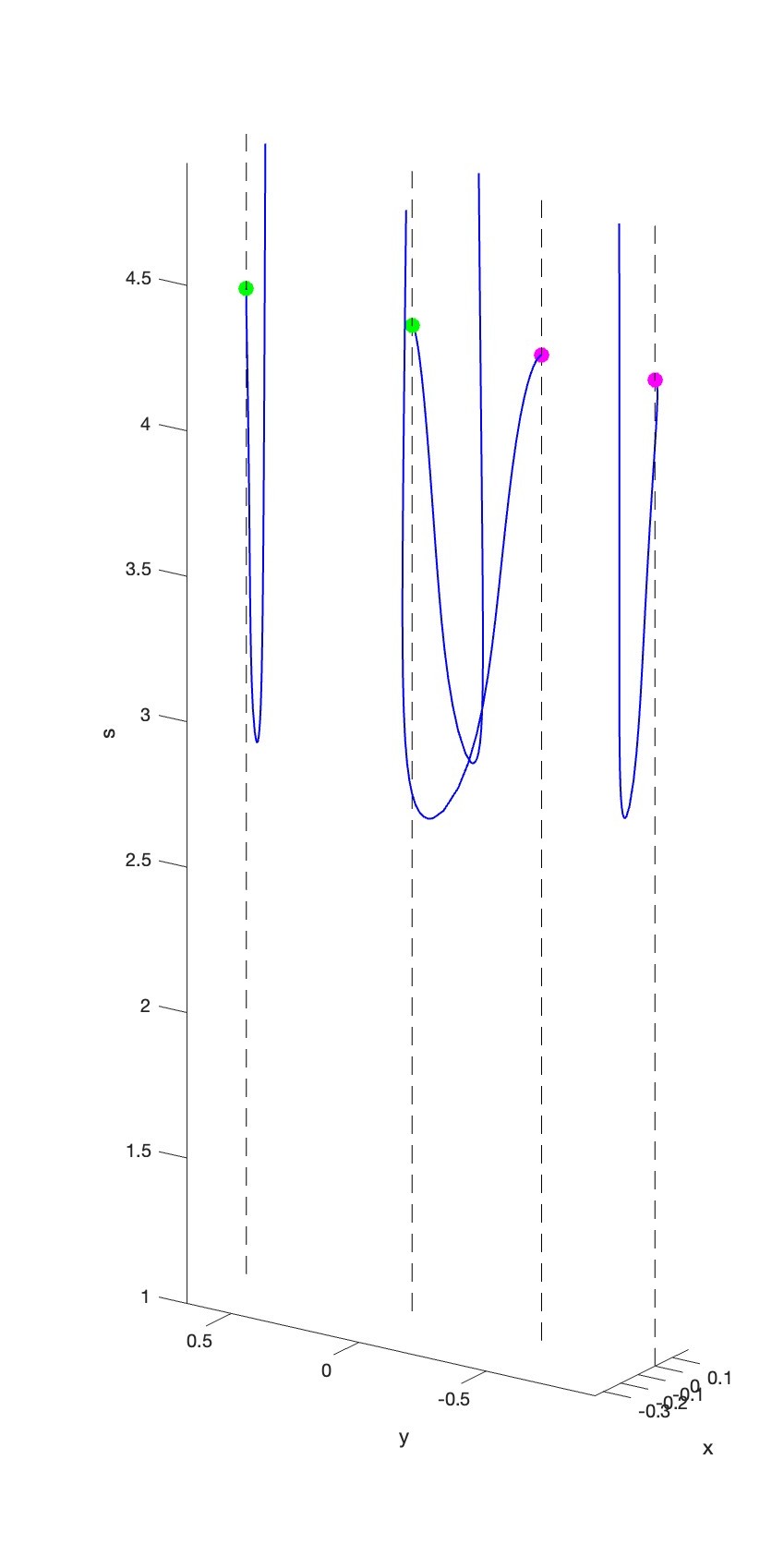} 
        \caption{Non-trivial branches for masses $[0.5,0.5,1,1]$ originating at $s_1 \approx 4.3$. The parameter $s$ decreases, reaching a turning point, and then increases again. }
        \label{fig_7}
    \end{minipage} \hfill
    \begin{minipage}{0.49\textwidth}
        \centering
        \vspace{0.67cm}
        \includegraphics[width=0.50\textwidth]{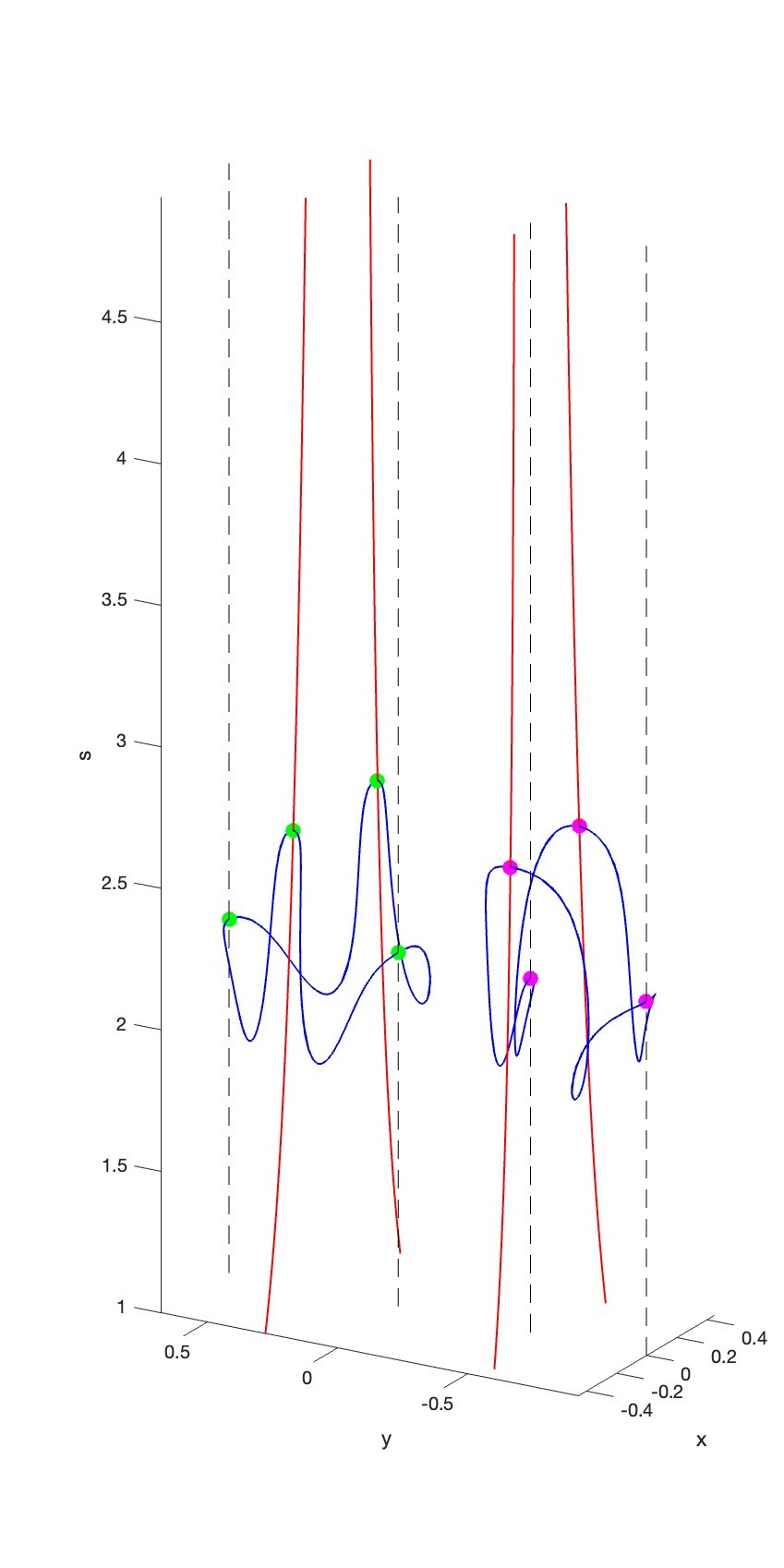} 
        \caption{Non-trivial branch for masses $[0.5,0.5,1,1]$ originating at $s_2 \approx 2.2$. Each pair of equal masses oscillate around each other, passing through different turning points. The red lines represent the branches detaching from one of the turning point.}
        \label{fig_8}
    \end{minipage}
\end{figure}

\noindent These configurations are initially minima of $U|_\s$, but at the turning point, we have a jump in the Morse index and they become saddle points. As before, no additional branches are found at this turning point. 
The branches originating from the second bifurcation point $s_2 \approx 2.2$ exhibit an intricate behavior, with each pair of equal masses oscillating around each other. Along the main branch, the parameter $s$ initially decreases and reaches a first turning point at $s_{\mathrm{turn},1} \approx 1.9$, where the Morse index jumps and no additional branches are found. The parameter then increases, reaching a second turning point at $s_{\mathrm{turn},2} \approx 2.7$, where the Morse index remains unchanged but two new branches detach. Along one of these bifurcating branches, $s$ increases, whereas along the other $s$ decreases, converging to the square central configuration as $s \to 1$. Continuing along the main branch, $s$ decreases again to the first turning point and then increases, returning to $s_2$, where the configuration is collinear, but with the two equal masses exchanged.

Finally, for $\alpha = 0.9$ (Figures~\ref{fig_9} and~\ref{fig_10}), we observe behaviors similar to those encountered previously. At $s_1 \approx 4.1$, the non-trivial branches follows a similar path to the branches originating from the first bifurcation point in the case of $\alpha = 0.5$. At $s_2 \approx 2.4$, we have a behavior similar to the one originating from the second bifurcation point in the equal masses case.

\begin{figure}[h!]
    \centering
    \begin{minipage}{0.49\textwidth}
        \centering
        \includegraphics[width=0.50\textwidth]{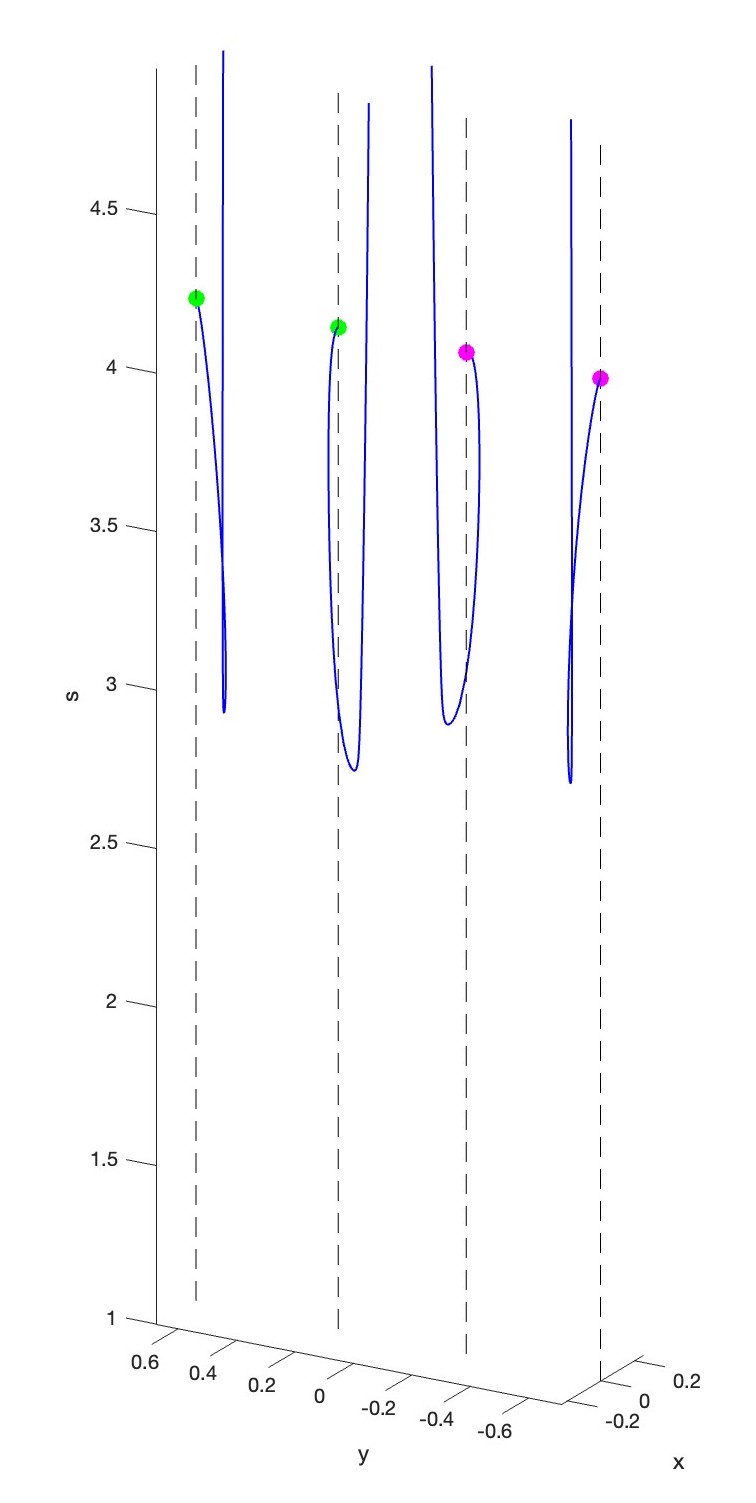} 
        \caption{Non-trivial branch for masses $[0.9,0.9,1,1]$ originating at $s_1 \approx 4.1$, behaving like the $\alpha=0.5$ case.}
        \label{fig_9}
    \end{minipage} \hfill
    \begin{minipage}{0.49\textwidth}
        \centering
        \includegraphics[width=0.50\textwidth]{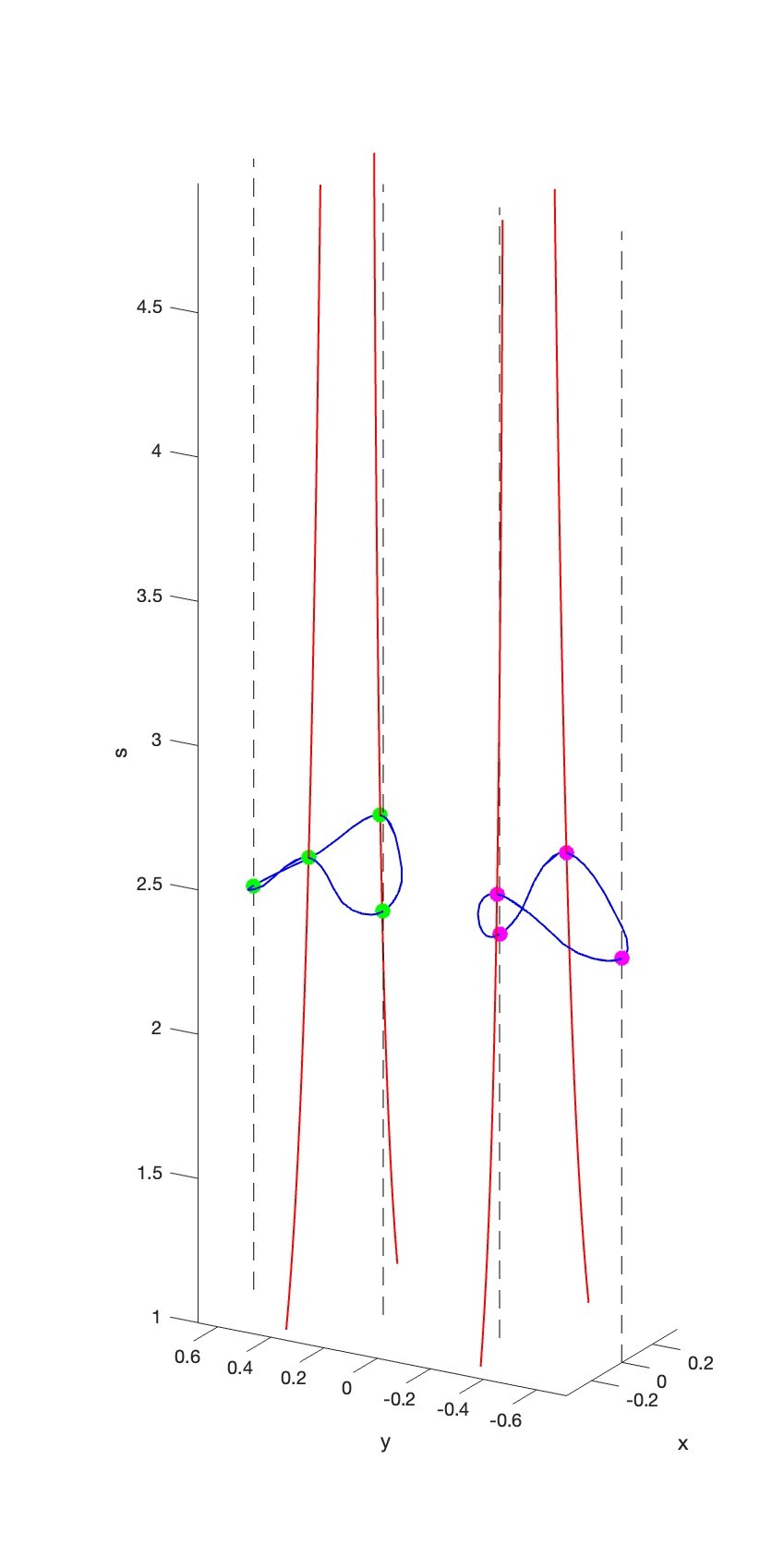} 
        \caption{Non-trivial branch for masses $[0.9,0.9,1,1]$ originating at $s_2 \approx 2.4$, resembling the equal mass case.}
        \label{fig_10}
    \end{minipage}
\end{figure}

These are just a few of the possible collinear configurations for the four-body problem. We have selected them to highlight some of the most interesting behaviors.

\begin{rmk}
The turning points encountered in these examples correspond to points at which the Jacobian is singular and the curve folds back with respect to the continuation parameter. To analyze these points in more details and identify multiple branches past bifurcation points, we use the \texttt{Matlab} toolbox \texttt{MatCont} \cite{Dhooge2003,Neirynck2019}.
Two distinct cases are observed. When the Morse index of the solutions changes at the turning point, the point is identified as a limit point (saddle-node bifurcation), and no additional branches are found (see Figures~\ref{fig_6},~\ref{fig_7}, and~\ref{fig_9}). On the other hand, when the Morse index does not change across the turning point, the point is detected as a branch point, and additional solution branches can be continued from it (see Figures~\ref{fig_3},~\ref{fig_4},~\ref{fig_5},~\ref{fig_8}, and~\ref{fig_10}).
 
\end{rmk}

\bibliographystyle{plainurl}
\bibliography{biblio} 

\begin{thebibliography}{10}

\bibitem{Albouy_Chen}
Alain Albouy and Alain Chenciner.
\newblock Le problème des $n$ corps et les distances mutuelles. (french)
  [{The} $n$-body problem and mutual distances].
\newblock {\em Inventiones Mathematicae}, 131(1):151--184, 1998.

\bibitem{Allgower1990}
Eugene~L. Allgower and Kurt Georg.
\newblock {\em Numerical Continuation Methods}.
\newblock Springer Series in Computational Mathematics. Springer Berlin
  Heidelberg, 1990.

\bibitem{AP}
Luca Asselle and Alessandro Portaluri.
\newblock Morse theory for {$S$}-balanced configurations in the {Newtonian}
  $n$-body problem.
\newblock {\em Journal of Dynamics and Differential Equations}, 35:907--946,
  2021.

\bibitem{APF}
Luca Asselle, Alessandro Portaluri, and Marco Fenucci.
\newblock Bifurcation of balanced configuration for the {Newtonian} $n$-body
  problem in $\mathbb{R}^4$.
\newblock {\em Journal of Fixed Point Theory and Applications}, 24(22), 2022.

\bibitem{Bredon}
Glen~E. Bredon.
\newblock {\em Introduction to Compact Transformation Groups}, volume~46 of
  {\em Pure and Applied Mathematics}.
\newblock Academic Press, New York, 1972.

\bibitem{CorvellecHantoute}
Jean-No{\"e}l Corvellec and Abderrahim Hantoute.
\newblock Homotopical stability of isolated critical points of continuous
  functionals.
\newblock {\em Set-Valued Analysis}, 10:143--164, 2002.
\newblock \href {https://doi.org/10.1023/A:1016544301594}
  {\path{doi:10.1023/A:1016544301594}}.

\bibitem{Dhooge2003}
A.~Dhooge, W.~Govaerts, and Yu.~A. Kuznetsov.
\newblock {MATCONT: A MATLAB package for numerical bifurcation analysis of
  ODEs}.
\newblock {\em ACM Trans. Math. Softw.}, 29(2):141–164, 2003.
\newblock \href {https://doi.org/10.1145/779359.779362}
  {\path{doi:10.1145/779359.779362}}.

\bibitem{SantosLeandro2026}
Michelle~G. dos Santos and Eduardo S.~G. Leandro.
\newblock Bifurcations of two symmetric classes of central configurations.
\newblock {\em Journal of Dynamics and Differential Equations}, 2026.
\newblock \href {https://doi.org/10.1007/s10884-026-10536-9}
  {\path{doi:10.1007/s10884-026-10536-9}}.

\bibitem{FPR}
Patrick~M. Fitzpatrick, Jacobo Pejsachowicz, and Lazaro Recht.
\newblock Spectral flow and bifurcation of critical points of
  strongly-indefinite functionals. part i: General theory.
\newblock {\em Journal of Functional Analysis}, 162(1):52--95, 1999.
\newblock \href {https://doi.org/10.1006/jfan.1998.3366}
  {\path{doi:10.1006/jfan.1998.3366}}.

\bibitem{GKS2023}
Anna Go{\l}{\k e}biewska, Joanna Kluczenko, and Piotr Stefaniak.
\newblock Bifurcations from degenerate orbits of solutions of nonlinear
  elliptic systems.
\newblock {\em Journal of Fixed Point Theory and Applications}, 25(1):36, 2023.
\newblock \href {https://doi.org/10.1007/s11784-022-01038-4}
  {\path{doi:10.1007/s11784-022-01038-4}}.

\bibitem{Kowalczyk2015}
Marta Kowalczyk.
\newblock Bifurcations of critical orbits of invariant potentials with
  applications to bifurcations of central configurations of the $n$-body
  problem.
\newblock {\em Nonlinear Analysis: Real World Applications}, 24:108--125, 2015.
\newblock \href {https://doi.org/10.1016/j.nonrwa.2015.01.006}
  {\path{doi:10.1016/j.nonrwa.2015.01.006}}.

\bibitem{LuParameterized}
Guangcun Lu.
\newblock Parameterized splitting theorems and bifurcations for potential
  operators, part i: Abstract theory.
\newblock {\em Discrete and Continuous Dynamical Systems}, 42(3):1243--1316,
  2022.
\newblock \href {https://doi.org/10.3934/dcds.2021154}
  {\path{doi:10.3934/dcds.2021154}}.

\bibitem{Moeckel1994}
Richard Moeckel.
\newblock Celestial mechanics (especially central configurations).
\newblock Handwritten lecture notes, Trieste, 1994.
\newblock URL: \url{http://www.math.umn.edu/~rmoeckel/notes/Notes.html}.

\bibitem{Moeckel2014}
Richard Moeckel.
\newblock Lectures on central configurations.
\newblock Lecture notes, Centre de Recerca Matemàtica, January 2014.
\newblock URL:
  \url{http://www.math.umn.edu/~rmoeckel/notes/CentralConfigurations.pdf}.

\bibitem{Neirynck2019}
{Neirynck, Niels}.
\newblock {\em {Advances in numerical bifurcation software : MatCont}}.
\newblock PhD thesis, {Ghent University}, {2019}.
\newblock URL: \url{{https://sourceforge.net/projects/matcont/}}.

\bibitem{PereraMorse}
Kanishka Perera.
\newblock Morse theory and applications to variational problems.
\newblock In David~Yang Gao and Dumitru Motreanu, editors, {\em Handbook of
  Nonconvex Analysis and Applications}, pages 475--506. International Press,
  Somerville, MA, 2010.

\bibitem{sim_relative_1978}
Carles Simó.
\newblock Relative equilibrium solutions in the four body problem.
\newblock {\em Celestial Mechanics}, 18(2):165--184, August 1978.

\end{thebibliography}

\end{document}